\documentclass[11pt,review,onefignum,onetabnum]{article}

\usepackage[a4paper,text={160.0truemm,240.0truemm},centering]{geometry}
\usepackage[normalem]{ulem}
\usepackage{amsmath}
\usepackage{amsthm}
\usepackage{hyperref}
\usepackage{cleveref}

\usepackage{todonotes}

\definecolor{colorJRred}{rgb}{1.,0.,0.}

\definecolor{colorJRblue}{rgb}{0.,0.,1.}

\definecolor{colorJRgreen}{rgb}{1.,0.,1.}

\usepackage{caption}
\newcommand{\bu}{\mathbf{u}}
\newcommand{\bue}{\overline{\mathbf{u}}_{\rm e}}
\newcommand{\ueq}{\overline{u}_{\rm e}}
\newcommand{\hu}{\widehat{u}}
\DeclareMathOperator{\sech}{sech}

\theoremstyle{remark}
\newtheorem{remark}{Remark}
\theoremstyle{hypothesis}
\newtheorem{hypothesis}{Assumption}
\newtheorem{theorem}{Theorem}
\newtheorem{lemma}[theorem]{Lemma}

\usepackage{lipsum}
\usepackage{amsfonts}
\usepackage{graphicx}
\usepackage{epstopdf}
\usepackage{algorithmic}
\ifpdf
  \DeclareGraphicsExtensions{.eps,.pdf,.png,.jpg}
\else
  \DeclareGraphicsExtensions{.eps}
\fi

\usepackage{enumitem}
\setlist[enumerate]{leftmargin=.5in}
\setlist[itemize]{leftmargin=.5in}

\crefname{hypothesis}{Hypothesis}{Hypotheses}

\title{Analysing transitions from a Turing instability to large periodic patterns in a reaction-diffusion system
}

\author{Christopher Brown\thanks{Department of Mathematics, University of Surrey, Guildford, GU2 7XH,
  UK.}
\and Gianne Derks$^*$\thanks{corresponding author: g.derks@surrey.ac.uk, g.l.a.derks@math.leidenuniv.nl}
\and Peter van Heijster{\thanks{Mathematical and Statistical Methods - Biometris, Wageningen University \& Research, 6708 PB, Wageningen, The Netherlands} }{\thanks{School of Mathematical Sciences, Queensland University
  of Technology, Brisbane, Queensland 4001, Australia.} }%
\and David J.B. Lloyd\footnotemark[1]}

\usepackage{amsopn}

\begin{document}

\maketitle

\begin{abstract}
Analytically tracking patterns emerging from a small amplitude Turing instability to large amplitude remains a challenge as no general theory exists. 
In this paper, we consider a three component reaction-diffusion system with one of its components singularly perturbed, this component is known as the fast variable. We develop an analytical theory describing the periodic patterns emerging from a Turing instability using geometric singular perturbation theory. 
We show analytically that after the initial Turing instability, spatially periodic patterns evolve into a small amplitude spike in the fast variable whose amplitude grows as one moves away from onset. 
This is followed by a secondary transition where the spike in the fast variable widens, its periodic pattern develops two sharp transitions between two flat states and the amplitudes of the other variables grow.
The final type of transition we uncover analytically is where the flat states of the fast variable develop structure in the periodic pattern. 
The analysis is illustrated and motivated by a numerical investigation.
We conclude with a preliminary numerical investigation where we uncover more complicated periodic patterns and snaking-like behaviour that are driven by the three transitions analysed in this paper. 
This paper provides a crucial step towards understanding how periodic patterns transition from a Turing instability to large amplitude.

\end{abstract}


\begin{AMS}
{34D15, 
34E15, 
35K40,
35K57, 
37J46 
}
\end{AMS}

\section{Introduction}

The emergence of periodic patterns is one of the simplest examples of
pattern formation and is often found in physical and biological systems such as crime hotspots~\cite{Buttenschoen2020,Lloyd2013,Wang2018}, vegetation patches~\cite{olfa2020},
cell polarization~\cite{Verschueren2017}, and plant root hair growth~\cite{Avitabile2018} to name but a few. These periodic patterns can be combined to form more complex patterns in two-dimensions
with sharp interfaces at the transition between the patterns
like so-called grain boundaries~\cite{Ercolani2018,hoyle06,Lloyd2017,Scheel2014,Zhang2021}. These are often seen in nature, for instance in Rayleigh-Benard convection~\cite{Haragus2021,hoyle06}, gannets nesting~\cite{Subramanian2021}, graphene~\cite{Hirvonen2016}, and phyllotaxis~\cite{Pennybacker2015}.

For small amplitude periodic patterns that emerge from a Turing instability, a generic
theory exists (see for instance~\cite{hoyle06} for an overview and
references) 
which yields insights into the various
behaviours one can expect to observe in experiments. 
In contrast, no general theory exists for far from
equilibrium patterns. Singular perturbation theory \cite{F79,J95,K99} is
one of the few techniques that has yielded new insights into the
complex phenomenology of patterns far from onset; see, for instance, \cite{Doelman2019, Doelman2012, HEK} and references therein. 
Recently, there has been an increasing interest in linking the small
amplitude {\it homoclinic} patterns, that emerge near a Turing
instability, to {\it localized} patterns 
found near the singular limit away from onset through a mixture of
numerical investigations, return-map analysis and singular
perturbation
theory~\cite{Saadi2021b,Saadi2022,Saadi2021,Champneys2021,Verschureren2021}. The
aim of this paper is to investigate analytically spatially {\it periodic} patterns emerging from a Turing instability using geometric singular perturbation theory. 

In particular, we study this connection
for 
a three-component reaction-diffusion system where one of the
components has a much smaller diffusion coefficient than the other two
which gives a spatially singular perturbed system.  This particular
model originated as a phenomenological model of gas-discharge
dynamics~\cite{P98,P14,P97}, see also \cite{liehr2013dissipative} and
references therein. It can be written as
\begin{equation}
\begin{array}{rcrcl}
U_t &=& \varepsilon^2\nabla^2 U  &+& U - U^3 - (A V + B W + C),\\
\tau V_t &=& \nabla^2V &+& U - V,\\
\theta W_t &=& D^2\nabla^2W &+& U - W,
\end{array}
\label{e:2d_system}
\end{equation}
where $(U,V,W)=(U,V,W)(x,t)\in \mathbb{R}^3$, with
$(x,t)\in\mathbb{R}^n\times\mathbb{R}^+$, $\nabla^2$ is the standard
Laplacian operator in
$\mathbb{R}^n$, 
$A, B, C, D\in\mathbb{R}$ and $\varepsilon,\tau,\theta\in\mathbb{R}^+$.
Typically, $\varepsilon$ is taken to be the small parameter in the
system and the parameters $\tau,\theta$ are assumed to be at least order~$1$ with respect to the small parameter $\varepsilon$, that is, $\tau,\theta = \mathcal{O}(\varepsilon^{-\chi}),$ for some $\chi \geq 0$.  Provided $D>\varepsilon$, we can assume, without a loss of
generality, that $D>1$.
We will restrict to patterns in one spatial dimension, i.e., $n=1$, and we will
consider the existence of stationary periodic patterns when
the parameters $B$ and $C$ are small (order $\varepsilon$) and the
parameter $A$ ranging from small to order~$1$. 
In~\cite{P98,P14,P97} versions of this model were studied to show the existence of Turing instabilities leading to the emergence of small amplitude spatially periodic patterns.
In addition, 
various results on the existence
and stability of far from equilibrium localized states -- for $A$, $B$, and $C$ order
$\varepsilon$ (with $\varepsilon$ small) -- have been proved
\cite{CBDvHR15,CBvHIR19,DvHK09,R13,TvH21,H18,vHDK08,vHDKP10,vHS11,vHS14}, but
less is known about far from equilibrium periodic patterns. Indeed, the only analytical results for periodic patterns can be found in \cite{H18} which restricts to the case that $A$, $B$, and $C$ are of order $\varepsilon$. 
Moreover, a detailed systematic numerical continuation study of versions
of the reaction-diffusion system~(\ref{e:2d_system}) has only been
performed for interacting pulses and two-dimensional spots \cite{P02, NTU03,nishiura2003scattering,P98,P97}.

Our short
numerical exploration with $A$ varying shows that there exist many
periodic patterns where $U$ displays both rapid changes and more
gradual evolution, while $V$ and $W$ only change gradually, see Figures~\ref{fig_overview} and~\ref{fig_overview22}. In this paper, we will show that these patterns can be 
described by the singular perturbation
theory of Fenichel (see for instance \cite{F79,HEK,J95,K99} and references
therein). In addition, we are able to analytically describe the transition from small amplitude periodic patterns created in a Turing bifurcation (that occurs when $A$ is order 1), to those found near the singular limit. Hence, we significantly extend the existence results in \cite{H18}. 
A key contribution of this paper is the delicate analysis of the
slow-fast structure which induces a slow flow in the fast $U$ variable
as depicted in panels {\textcircled{\raisebox{-.9pt} {6}}}
-- {\textcircled{\raisebox{-.9pt} {8}}} in Figure~\ref{fig_overview}.\footnote{The fast variable is labelled $u$ in Figure~\ref{fig_overview} to indicate that we are simulating the time-independent version of (\ref{e:2d_system}), see also~(\ref{e:2d_system_ODE}).}

The paper is outlined as follows: in \cref{S:NUM} we start with a
numerical exploration of the periodic patterns of~\eqref{e:2d_system}
with $n=1$, $B$ and $C$ small and $A$ varying between
order $\varepsilon$ and order~$1$ to motivate the analysis. The
patterns observed display some fast transitions interspersing more
gradual behaviour. The major results of this paper, Theorems~\ref{th.one}-\ref{th.two_2}, related to the existence of three different types of numerically observed patterns are also given in \cref{S:NUM}.
Next, in \cref{S:SETUP}, we describe the two
spatial scales for the system (induced by
the singular nature of~\eqref{e:2d_system}), discuss the equilibria,
and introduce the slow-fast structures in the system. In this section, we also derive conditions for
the Turing
bifurcation from which near-equilibrium stationary periodic patterns
emerge, see Lemma~\ref{L:HH}. For this bifurcation to occur it is necessary that {\it not} all
three system parameters $A,B$ and $C$ are small. From this section
onwards, we focus predominantly on system parameter $A$, while we keep
the other parameters $B$ and $C$ small.  In \cref{S:SF}, we analyse the
far-from equilibria periodic patterns that have one slow-fast
transition (see, e.g., panels
\raisebox{.5pt}{\textcircled{\raisebox{-.9pt}
    {1}}} and \raisebox{.5pt}{\textcircled{\raisebox{-.9pt} {2}}} of
Figure~\ref{fig_overview}). 
Emerging from a Turing bifurcation, these patterns can only occur if $A$ is
order $1$ (since $B$ and $C$ are fixed at order~$\varepsilon$ values). That is, we prove Theorem~\ref{th.one}.
In \cref{S:SFS}, we analyse
the far-from equilibria periodic patterns that have two distinct
slow-fast transitions. Two types of patterns emerge. One, related to Theorem~\ref{th.two}, occurs
when all three system parameters $A,B$ and $C$ are small (see, e.g.,
panels \raisebox{.5pt}{\textcircled{\raisebox{-.9pt} {4}}} and
\raisebox{.5pt}{\textcircled{\raisebox{-.9pt} {5}}} of
Figure~\ref{fig_overview}) and is also studied in~\cite{H18}.
The other
pattern related to Theorem~\ref{th.two_2} is new and has not been studied yet. It involves the analysis of the slow evolution of the fast variable $U$; see, e.g., panels
\raisebox{.5pt}{\textcircled{\raisebox{-.9pt} {6}}} and
\raisebox{.5pt}{\textcircled{\raisebox{-.9pt} {7}}} of
Figure~\ref{fig_overview}.  We end the paper with a discussion and
outlook on further work.

\section{A short numerical exploration of \eqref{e:2d_system} and existence results}
\label{S:NUM}
As we focus on stationary one-dimensional periodic patterns
in~(\ref{e:2d_system}), we numerically investigate the time-independent
version of~(\ref{e:2d_system}) with $n=1$: 
\begin{equation}
\begin{array}{rcrcl}
0 &=& \varepsilon^2 u_{xx}  &+& u - u^3 - (A v + B w + C),\\
0 &=& v_{xx} &+& u - v,\\
0 &=& D^2 w_{xx} &+& u - w,
\end{array}
\label{e:2d_system_ODE}
\end{equation}
and periodic boundary
conditions, using the numerical continuation program AUTO-07P \cite{auto}.
More specifically, we continue in system parameter
$A$ with the other system
parameters fixed. 
We take first $B$ small and $C=0$ (and $D=3$, $\varepsilon=0.01$). 
We observe that a near-equilibrium pulse with small
amplitude and small width is born for $A$ of order $1$. 
In the next section we will show that this happens near $A \approx 2/3$ (indicated by the black dashed line in Figure~\ref{fig_overview}).
Upon
reducing $A$ the amplitude of the $u$ variable of this pulse grows to
approximately $2$,
however, its width stays small. The
$v$ and $w$ variables stay near-constant, see panels \raisebox{.5pt}{\textcircled{\raisebox{-.9pt}
    {1}}}-\raisebox{.5pt}{\textcircled{\raisebox{-.9pt}
    {3}}} of Figure~\ref{fig_overview}.  
The small width pulse can
be seen as slow-fast transition, hence this solution can be described
as a periodic solution with one slow-fast transition (in the $u$
variable). An analytical description of
this type of pulses and its existence interval will be derived in
section~\ref{S:SF}.

\begin{figure}[h!]
  \centerline{\includegraphics[width=1\textwidth]{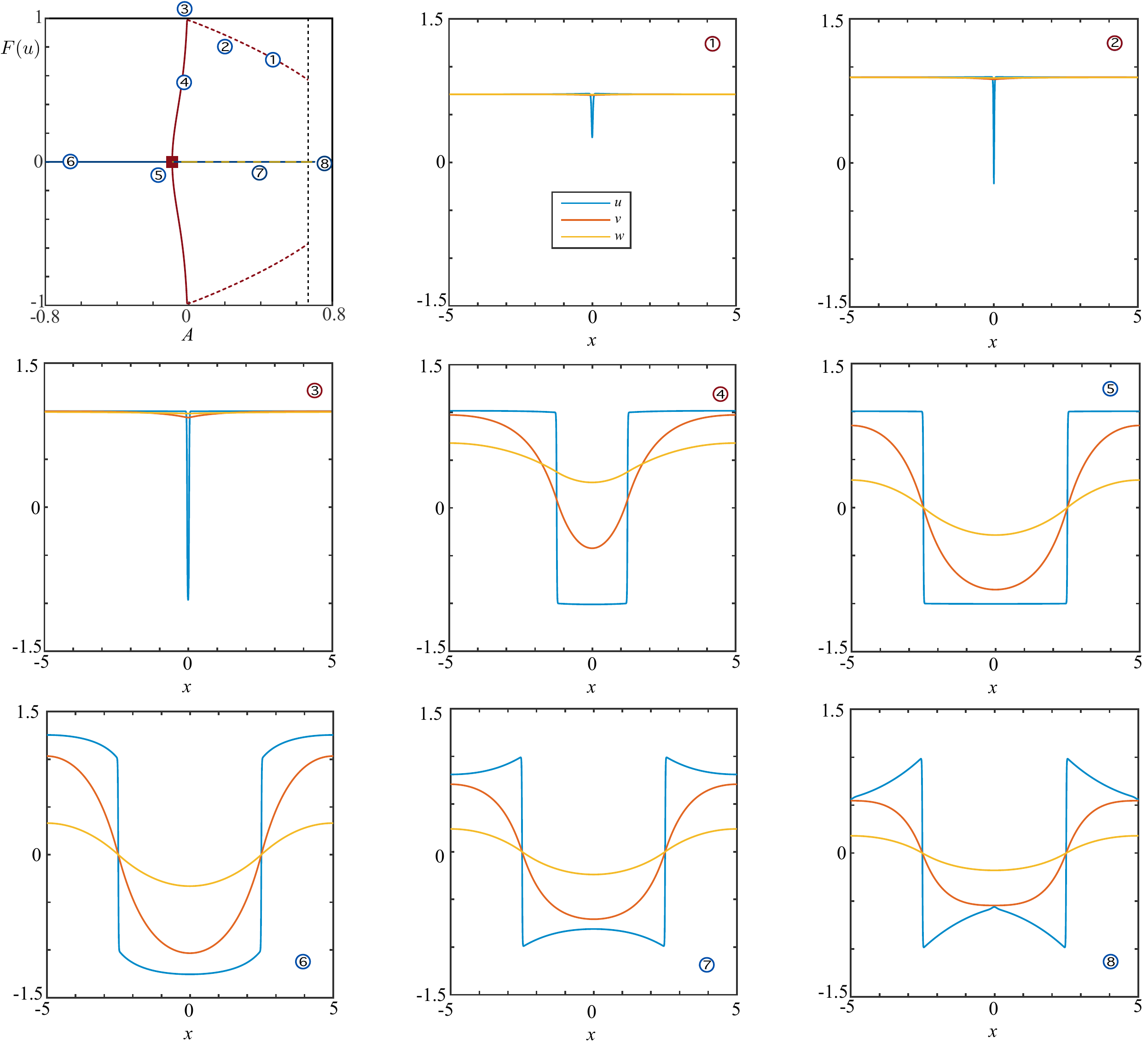}}
 \caption{
Numerical continuation in $A$ of 
\eqref{e:2d_system_ODE} on a domain of length $10$. The other
system parameters are kept fixed at $B=\varepsilon = 0.01$, $C=0$ and
$D=3$. The top left panel shows the obtained bifurcation diagram of
$A$ versus the integral of $u$ (its mass). The dashed black line indicates $A=2/3$. Panels \textnormal{\raisebox{.5pt}{\textcircled{\raisebox{-.9pt} {1}}}-\raisebox{.5pt}{\textcircled{\raisebox{-.9pt} {8}}}} show the associated profiles of the periodic patterns as indicated in the bifurcation diagram. }\label{fig_overview}
\end{figure}

Once $A$ is of order $\varepsilon$, the amplitude of the $u$-pulse stays
approximately the same upon further decreasing $A$, however, its
widths grows to order $1$ and the $u$-pulse transforms to a
front-back-like structure, see panels
\raisebox{.5pt}{\textcircled{\raisebox{-.9pt}
    {3}}}-\raisebox{.5pt}{\textcircled{\raisebox{-.9pt} {5}}} of
Figure~\ref{fig_overview}. The front-back structure corresponds to two
slow-fast transitions and this type of solution is studied in
section~\ref{S:SFS_Asmall}. As we have taken $C=0$ and the system thus has an
additional reflection symmetry, the pattern undergoes a pitchfork
bifurcation (shown in \raisebox{.5pt}{\textcircled{\raisebox{-.9pt}{5}}}) when the front-back structure has become symmetrically
spaced and $v=w=0$ at the fast transitions in $u$. When $C\neq 0$, but
small, the system loses this symmetry and the pitchfork bifurcation
becomes a saddle-node bifurcation, see Figure~\ref{fig_overview22}.

\begin{figure}[ht!]
  \centerline{\includegraphics[width=1\textwidth]{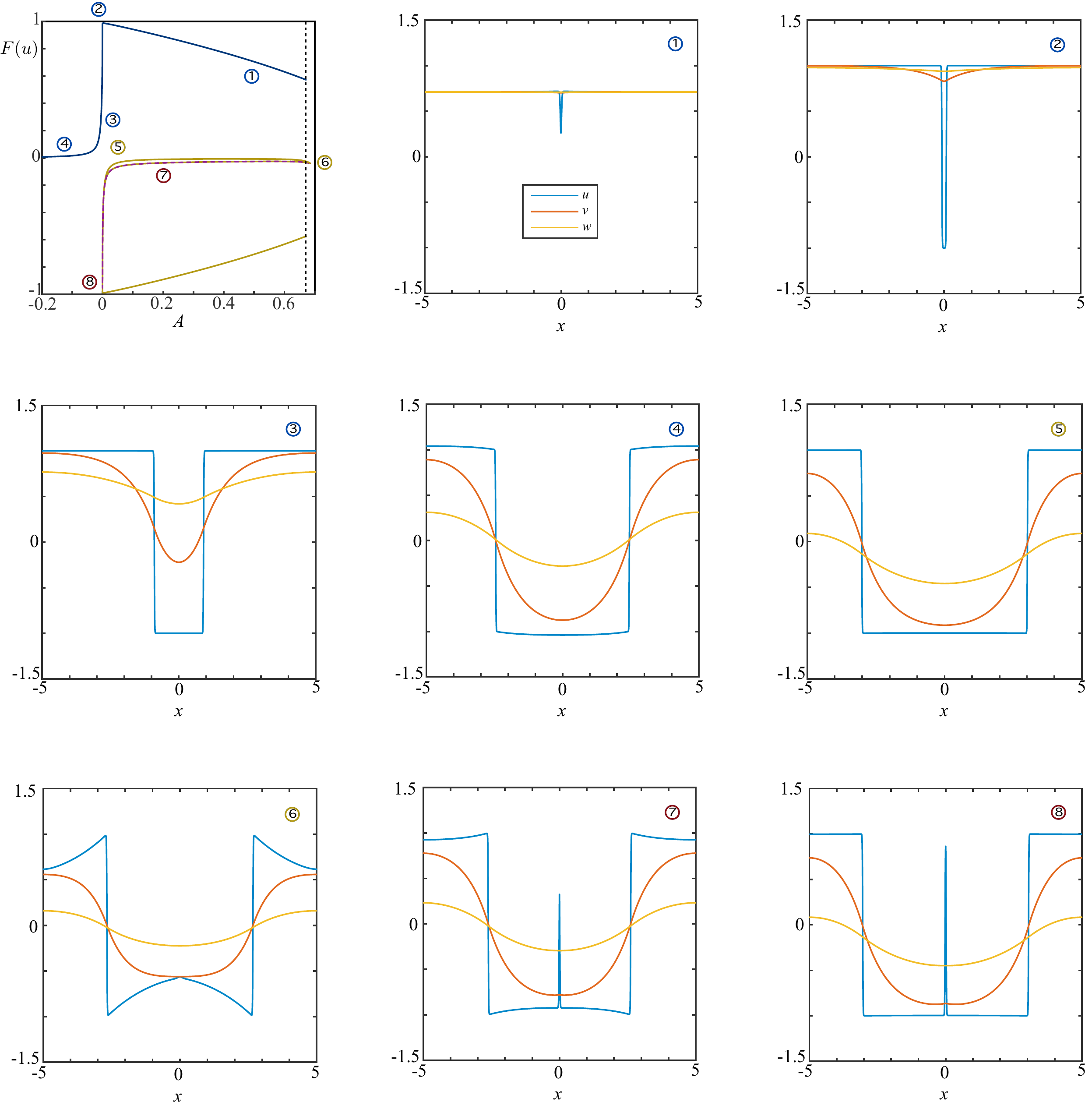}}
   \caption{
Numerical continuation in $A$ of 
\eqref{e:2d_system_ODE} on a domain of length $10$. The other
system parameters are kept fixed at $B=C=\varepsilon = 0.01$, and
$D=3$. The top left panel shows the obtained bifurcation diagram of
$A$ versus the integral of $u$ (its mass).
The dashed black line indicated $A=2/3$.
Panels \textnormal{\raisebox{.5pt}{\textcircled{\raisebox{-.9pt} {1}}}-\raisebox{.5pt}{\textcircled{\raisebox{-.9pt} {8}}}} show the associated profiles of the periodic patterns as indicates in the bifurcation diagram. }\label{fig_overview22}
\end{figure}   
Switching branches and changing $A$ further develops the symmetrically
spaced solution, this solution is studied in
section~\ref{S:SFS_Alarge}. For $A$ decreasing (i.e., $A$ negative),
this branch can be followed for a large range of $A$ values, see panel
\raisebox{.5pt}{\textcircled{\raisebox{-.9pt}
    {6}}} of
Figure~\ref{fig_overview}. For
increasing $A$ values (i.e., $A$ positive) back to
order $1$, the branch terminates (for an $A$ value larger than 2/3, see section~\ref{S:SFS_Alarge}) and we observe the creation of a new
near-equilibrium pulse with small amplitude and small width, that is,
we observe a pulse-splitting phenomena, see panels
\raisebox{.5pt}{\textcircled{\raisebox{-.9pt} {7}}} and \raisebox{.5pt}{\textcircled{\raisebox{-.9pt} {8}}} of
Figure~\ref{fig_overview} and, in particular, panels \raisebox{.5pt}{\textcircled{\raisebox{-.9pt} {6}}}-\raisebox{.5pt}{\textcircled{\raisebox{-.9pt} {8}}} of
Figure~\ref{fig_overview22}.

During the numerical exploration of (\ref{e:2d_system_ODE}) we continued in system parameter $A$, while assuming that the system parameter $B$ was small. Continuing in system parameter $B$, with $A$ small, results in similar bifurcation diagrams (results not shown). We hypothesise that this is (likely) due to the strong similarity in the structure of the linear slow components $V$ and $W$ in (\ref{e:2d_system}) (or $v$ and $w$ in (\ref{e:2d_system_ODE})), see also \cite{DvHK09,vHDK08,vHDKP10}. We do not further investigate this claim in this paper.

In the following sections, we will characterize these numerically observed patterns
analytically and determine their regions of
existence. We will not do this in the most generic case, instead we make the assumption that the system parameters $B$ and $C$ are small with respect to $\varepsilon$ (as we also did in the numerical exploration before).
\begin{hypothesis}\label{H1}
Assume $B, C$ of \eqref{e:2d_system} or equivalently \eqref{e:2d_system_ODE} are fixed and $\mathcal{O}(\varepsilon)$. In particular, assume that $B=\varepsilon B_1$ and $C=\varepsilon C_1$ with $B_1, C_1\in \mathbb{R}$ and $\mathcal{O}(1)$.
\end{hypothesis}
First we state the result about the existence of solutions as depicted in panels
\raisebox{.5pt}{\textcircled{\raisebox{-.9pt}
    {1}}} and \raisebox{.5pt}{\textcircled{\raisebox{-.9pt} {2}}} of
Figure~\ref{fig_overview} and panel
\raisebox{.5pt}{\textcircled{\raisebox{-.9pt}
    {1}}} of
Figure~\ref{fig_overview22}.
\begin{theorem}\label{th.one}
 Let Assumption~\ref{H1} hold 
 and $L>0$. Then there is
  some $\varepsilon_0>0$ such that for all $0<\varepsilon<\varepsilon_0$, the
  three-component reaction-diffusion system~\eqref{e:2d_system} has a stationary
  $2L$-periodic solution with one fast transition. The slow solutions are given by
  \[
    v_s(x) = \pm \sqrt{1-A} +\mathcal{O}(\varepsilon),\quad
    w_s(x)= \pm \sqrt{1-A} +\mathcal{O}(\varepsilon),\quad x\in [-L,L];
   \]
   and the fast solution
   \[
     u_f(\xi) = \pm u_h(\xi;A)+\mathcal{O}(\varepsilon)\,,
     \quad \xi = \dfrac{x}{\varepsilon} \in \left[-\frac{L}{\varepsilon},\frac{L}{\varepsilon}\right]\,.
   \]
   Here, $u_h(\xi;\sqrt{1-A})$ is the solution homoclinic to $\sqrt{1-A}$ in
   $u_{\xi\xi}+u-u^3-A\sqrt{1-A}=0$. 
\end{theorem}
Next we state the existence result for solutions as depicted in the panels
\raisebox{.5pt}{\textcircled{\raisebox{-.9pt}
    {4}}} and \raisebox{.5pt}{\textcircled{\raisebox{-.9pt} {5}}} of
Figure~\ref{fig_overview} and panel
\raisebox{.5pt}{\textcircled{\raisebox{-.9pt}
    {3}}} of
Figure~\ref{fig_overview22}.
\begin{theorem}\label{th.two}
Let Assumption~\ref{H1} hold and let $A = \varepsilon A_1, $ with $A_1 \in \mathbb{R}$ and $\mathcal{O}(1)$, $D>1$  
 and $L>0$
 be such 
 that the Melnikov condition 
 \begin{equation}\label{e:jump}
  M(L-2x^{**}) +C_1 = 0,
  \quad 0<x^{**}<L;\quad\text{ where }
  M(z) = A_1\,\frac{\sinh(z)}{\sinh(L)} +
B_1\,\frac{\sinh(z/D)}{\sinh(L/D)}
\end{equation}
 has $N \in \{1, 2, 3\}$ solutions $x^{**}_i, i=1, \ldots, N$.
 Then there is
  some $\varepsilon_0>0$ such that for all $0<\varepsilon<\varepsilon_0$, the
  three-component reaction-diffusion system~\eqref{e:2d_system}  has $N$ stationary
  $2L$-periodic solutions with two fast transitions. The slow solutions are in lowest order given by \eqref{e:I1}--\eqref{e:I2} (with $x^{**} = x^{**}_i$) in Appendix~\ref{s:slow_approx_params_small} and the fast solutions are in lowest order given by
\begin{equation}\label{FAST_new}
u_{f}^\pm(x) =1 + \tanh\left(\frac{x-x^{**}}{\sqrt{2} \varepsilon} \right) - \tanh\left(\frac{x+x^{**}}{ \sqrt{2}\varepsilon} \right) + \mathcal{O}(\varepsilon)\,, \quad x \in [-L,L].
\end{equation}  
 The co-periodic stability, related to perturbations with the same period, is given by 
\begin{equation}\label{e:stab_per}
\left.\frac{dM(z)}{dz}\right|_{z=L-2x^{**}} < 0.
\end{equation}
   \end{theorem}
As part of the proof it will be shown that the Melnikov condition \eqref{e:jump} has 0, 1, 2, or 3 solutions.

Finally we state the existence results for the periodic solutions as depicted in the panels
\raisebox{.5pt}{\textcircled{\raisebox{-.9pt}
    {6}}} and \raisebox{.5pt}{\textcircled{\raisebox{-.9pt} {7}}} of
Figure~\ref{fig_overview} and panel
\raisebox{.5pt}{\textcircled{\raisebox{-.9pt}
    {4}}} of
Figure~\ref{fig_overview22}.
 \begin{theorem}\label{th.two_2}
Let Assumption~\ref{H1} hold  and let $A = A_0 + \varepsilon A_1, $ with $A_0 \neq 0$, $D>1$ 
 and $L>0$.
If $A_0 < 2/3$  or 
    \begin{align} \label{LMAX_new}
 A_0 \geq 2/3 \quad\mbox{and}\quad L< L_{\rm max}(A_0) :=  
      6\log\left(\frac{\sqrt6+\sqrt{9A_0-2}}{\sqrt2+\sqrt{9A_0-6}}\right),
    \end{align}
then there is
  some $\varepsilon_0>0$ such that for all $0<\varepsilon<\varepsilon_0$, the
  three-component reaction-diffusion system~\eqref{e:2d_system} has a stationary
  $2L$-periodic solution with two fast transitions at $\pm L/2$. 
  
  The fast solution is in lowest order given by
  $$
  u_{f}^\pm(x) = \mp \tanh\left(\frac{x\pm L/2}{\sqrt{2} \varepsilon} \right) + \mathcal{O}(\varepsilon)\
  $$
  near $x=\mp L/2$
and, away from $x=\mp L/2$, by the solutions of
\begin{equation}\label{eq.uq}
  u_x  = - \frac{A_0q}{3u^2-1}, \quad
  q_x = \frac{(1-A_0)u-u^3}{A_0} ;\quad |u| >  \frac1{\sqrt3},
\end{equation}
with the boundary conditions $u \to \pm 1$ for $x \to -L/2_\mp$ (where the subscript $\pm$ denotes the left and right limit) and $u \to \pm 1$ for $x \to L/2_\pm$.
  
The slow $v$ solution is in lowest order given by the solution to \eqref{eq.uq} and a bijective relation $u=u_0^{\pm}(v)$ between
$\left\{v \mid \pm A_0v \leq 2/(3\sqrt3) \right\}$ and
$\left\{ u \mid \pm u \geq 1/\sqrt3\right\}$ (implicitly given by $u_0^{\pm} - (u_0^{\pm})^3 = A_0 v$ with $A_0 v < 2/(3\sqrt3 )$ for $u_0^+$ and $A_0 v > -2/(3\sqrt3 )$ for $u_0^-$).  The slow $w$ solution is in lowest order given by 
\begin{align}
\label{slowTH2}
\begin{aligned}
  w(x) &= & \frac{\cosh\left(x/D\right)}{D \cosh\left(L/(2D)\right)}
  \int_0^{\frac L2} \sinh\left(\frac{L-2\xi}{2D}\right)\, u(\xi) \,
  d\xi - \frac 1D  \int_0^{x} \sinh\left(\frac{x-\xi}D\right)\, u(\xi) \,d\xi + \mathcal{O}(\varepsilon)\,, \\
   r(x) &= & \frac{\sinh\left(x/D\right)}{D \cosh\left(L/(2D)\right)}
  \int_0^{\frac L2} \sinh\left(\frac{L-2\xi}{2D}\right)\, u(\xi) \,
  d\xi - \frac 1D  \int_0^{x} \cosh\left(\frac{x-\xi}D\right)\, u(\xi) \,d\xi + \mathcal{O}(\varepsilon)\,,
  \end{aligned}
\end{align}
for $x \in
\left(- L/2, L/2\right)$ and
with $u$ the solution of \eqref{eq.uq} in the same interval.
On the other two intervals, $w$ and $r$ are to leading order given by symmetry
\begin{align}
\label{slowTH22}
\begin{aligned}
  w(x)&=-w\left(x-L\right),   \quad r(x)=-r\left(x-L\right),
  \quad \mbox{for} \quad x\in  \left(\frac L2, L\right),  \\
  w(x)&=-w\left(x+L\right),   \quad r(x)=-r\left(x+ L\right), \quad
  \mbox{for} \quad x\in  \left(-L, -\frac L2\right).  
  \end{aligned}
\end{align}
\end{theorem}   
Before we prove these theorems, first we will make explicit the slow-fast structure of the model 
and analyse its basic properties in the next section.
\begin{remark} \label{XX}
For the scaling of Theorem~\ref{th.two} various results on the existence and
stability of localized states have been proved
\cite{CBDvHR15,CBvHIR19,DvHK09,NTU03,R13,TvH21,H18,vHDK08,vHDKP10,vHS11,vHS14}.
Less is known for periodic patterns, although in~\cite{H18} an action
functional approach was used to determine criteria for existence and stability 
of stationary $2L$-periodic solutions with two fast transitions, i.e., Theorem~\ref{th.two} was derived and proved. For completeness of the current paper, we also derive this
existence condition with our methodology. Subsequently, we go beyond~\cite{H18} and further
investigate this condition in the context of the transitions between
patterns. Note that in the aforementioned works a slightly different notation is used 
\[\alpha  \leftrightarrow A_1\,, \qquad \beta \leftrightarrow B_1\,, \qquad \gamma \leftrightarrow C_1,
\]
and the periodic solutions constructed in \cite{H18} are, compared to the periodic solutions constructed here, {\emph{mirrored}} in the $x$-axis, see the right panel of Figure~\ref{f:periodic_setup}.  
\end{remark}

\section{Equilibria, the Turing bifurcation, and the slow-fast structure}
\label{S:SETUP}
\subsection{Singular limit set-up}
We write the time-independent system~(\ref{e:2d_system_ODE})
as a first-order system of ordinary differential equations (ODEs) by
introducing $p:= \varepsilon u_x , q:= v_x, r:=D w_x$. This gives the system
\begin{equation}
\begin{array}{rcl}
\varepsilon u_x &=& p,\\
\varepsilon p_x &=& -u + u^3 + (A_0+\varepsilon A_1) v+(B_0+\varepsilon B_1) w
                 +C_0+\varepsilon C_1 ,\\
v_x &=& q,\\
q_x &=&v - u, \\
w_x &=& \dfrac{r}{D},\\[3mm]
r_x &=& \dfrac{1}{D}(w - u), 
\end{array}
\label{e:6slow}
\end{equation}
where we used a regular expansion for $A=A_0+\varepsilon A_1$, $B=B_0+\varepsilon B_1$, and
$C=C_0+\varepsilon C_1$ to be able to easily distinguish between system parameters of strict order~$1$ and order~$\varepsilon$.
Given the singular perturbed nature of the ODE
system~(\ref{e:6slow}), this system can be viewed as the {\it slow system}
with the corresponding {\it fast system} of the form 
\begin{equation}
\begin{array}{rcl}
u_\xi &=& p,\\
 p_\xi &=& 
-u + u^3 + (A_0+\varepsilon A_1) v+(B_0+\varepsilon B_1) w
                 +C_0+\varepsilon C_1, \\
v_\xi &=& \varepsilon q,\\
q_\xi &=&\varepsilon(v - u), \\[1.5mm]
w_\xi &=& \varepsilon\dfrac{r}{D},\\[3mm]
r_\xi &=& \dfrac{\varepsilon}{D}(w - u), 
\end{array}
\label{e:6fast}
\end{equation}
where the fast variable is written as $\xi = x/\varepsilon$. 
For $\varepsilon\neq0$, these systems are equivalent, though they are different in the singular limit $\varepsilon \to 0$.

The fast system is Hamiltonian as it takes the form $\bu_\xi = J \nabla
H(\bu)$, with $\bu = (u,p,v,q,w,\\ r)$ and Hamiltonian
\begin{equation}
  \label{eq.ham}
  H(\bu) = \frac12{p^2} - \frac14{u^4}+\frac12u^2 - \left(
  \frac{A q^2}{2}+ \frac{B r^2}{2}
  - \frac{A v^2}{2}- \frac{B w^2}{2}
  +(A v + B w +C)u  \right),
  \end{equation}
and nonstandard symplectic operator
\[
J = \mathop{\rm diag}\left(
  \begin{pmatrix}  0&1\\-1&0 \end{pmatrix},
  \begin{pmatrix}  0&-\dfrac\varepsilon A\\\dfrac\varepsilon A&0 \end{pmatrix},
  \begin{pmatrix}  0&-\dfrac\varepsilon{DB}\\\dfrac\varepsilon{DB}&0 \end{pmatrix}
\right),
\]
where we recall that $A=A_0+\varepsilon A_1$, $B=B_0+\varepsilon B_1$, and
$C=C_0+\varepsilon C_1$. 

\subsection{Equilibria for $\varepsilon \neq 0$}\label{s:equilbria}

For $\varepsilon \neq 0$, the equilibria of \eqref{e:6slow}/\eqref{e:6fast} are given by
\begin{equation}\label{eq:equilib}
\bue= \ueq\,(1,0,1,0,1,0),
\quad\mbox{with}\quad
\ueq^3-\ueq(1-A-B) +C =0\,.
\end{equation}
If $A+B\geq 1$, or if $A+B<1$ and
$|C| > 2((1-(A+B))/3)^{3/2}$,
then the equation for
$\ueq$ has exactly one solution. If $A+B<1$ and
$|C| < 2((1-(A+B))/3)^{3/2}$,
then there are exactly
three roots, see Figure~\ref{fig_ueps}. 
\begin{figure}
  \centerline{\includegraphics{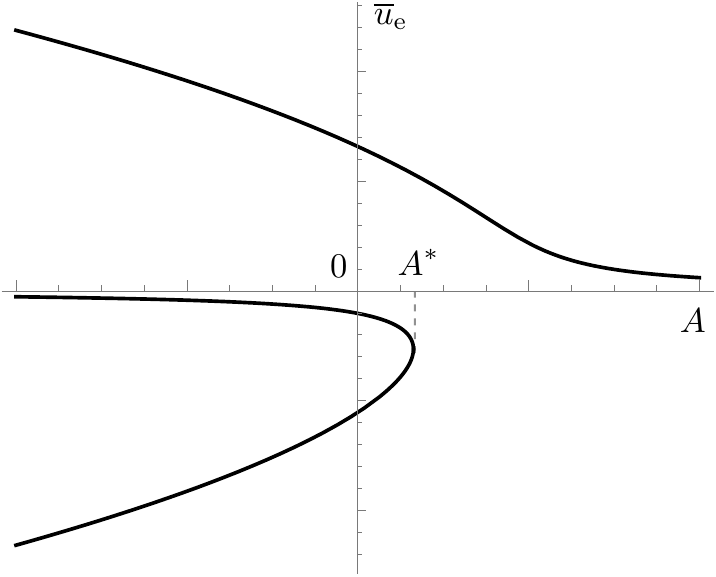}}
  \caption{
    Typical fixed point structure of
   \eqref{eq:equilib} for fixed $B$ and $C$; upon increasing $A$ we transition from three
   roots to one root at $A=A^*:=1-B-3(C/2)^{2/3}$.
 }
   \label{fig_ueps}
\end{figure}

The characteristic polynomial associated with the
linearisation about a fixed point in~\eqref{e:6fast} is
\begin{equation}\label{eq:charpoly}
p(\lambda) =
(\lambda^2-(3\ueq^2-1))(\lambda^2-\varepsilon^2)(D^2\lambda^2-\varepsilon^2) 
+ A\varepsilon^2(D^2\lambda^2-\varepsilon^2) +
BD^2\varepsilon^2(\lambda^2-\varepsilon^2).
\end{equation}
If $3\ueq^2-1$ is not small, that is, not of order~$o(1)$ for
$\varepsilon\to0$, then there are two fast eigenvalues with
$\lambda^2_f (\varepsilon)= 3\ueq^2-1 +\mathcal{O}(\varepsilon^2)$ and four
slow ones. The slow ones can be written as
$\lambda^2_s(\varepsilon)=\lambda_0\varepsilon^2 +\mathcal{O}(\varepsilon^4)$,
where $\lambda_0$ is a solution of
\[
D^2(3\ueq^2-1)\lambda_0^4 -\lambda_0^2 [(1+D^2)(3\ueq^2-1)+D^2A+B] 
+ (3\ueq^2-1)-(A+B)=0.
\]
If $3\ueq^2-1$ becomes small, i.e., when $A+B+\sqrt{3} |C| -2\slash3 =
o(1)$, then the eigenvalues denoted by $\lambda^2_f
(\varepsilon)$ become small too and the fast-slow decomposition breaks
down. In the following sections, we will see that this is consistent
with the slow manifold losing hyperbolicity at this point.

To finish this section on the equilibria, 
we give some more details about the persisting fixed points in the
full system when
$B_0=0=C_0$, as these equilibria will play an important role in the upcoming
analysis. 
\begin{lemma}
\label{L:new1}
Let Assumption~\ref{H1} hold. 
Then there is some $\varepsilon_0>0$ such that for all $0<\varepsilon<\varepsilon_0$ the equilibria of 
\eqref{e:6slow}/\eqref{e:6fast} given in \eqref{eq:equilib} satisfy $\ueq =
\ueq^0(A_0)+\mathcal{O}(\varepsilon)$ where
$\ueq^0(A_0)$ solves 
\begin{align}
\label{C0}
(\ueq^0)^3
-\ueq^0(1-A_0)=0\,.
\end{align}
\begin{itemize}
\item For $A_0>1$ \eqref{e:6slow}/\eqref{e:6fast} has one equilibrium and this equilibrium is $\mathcal{O}(\varepsilon)$. Moreover, the linearisation about the equilibrium has two real hyperbolic slow $\mathcal{O}(\varepsilon)$ eigenvalues and two pairs of purely
imaginary eigenvalues, one fast $\mathcal{O}(1)$ pair and one slow $\mathcal{O}(\varepsilon)$ pair; 
\item for $A_0<1$ \eqref{e:6slow}/\eqref{e:6fast} has three equilibria, one of which is $\mathcal{O}(\varepsilon)$ while the other two are 
$\mathcal{O}(1)$. The linearisation about the small equilibrium has two pairs of real hyperbolic slow $\mathcal{O}(\varepsilon)$ eigenvalues and one pair of purely imaginary fast $\mathcal{O}(1)$ eigenvalues. 
The linearisation about the two $\mathcal{O}(1)$ equilibria have
\begin{itemize}
\item for $A_0<2/3$ two pairs of hyperbolic slow $\mathcal{O}(\varepsilon)$ eigenvalues and one pair of hyperbolic fast $\mathcal{O}(1)$ eigenvalues. At lowest order they are given by
$\lambda_{s,1}^2=\varepsilon^2/D^2$;
$\lambda_{s,2}^2=2\varepsilon^2(1-A_0)/(2-3A_0)$; $\lambda_f^2=2-3A_0$; and
\item  
for $2/3<A_0<1$ two pairs of hyperbolic slow $\mathcal{O}(\varepsilon)$  eigenvalues and one pair of purely imaginary fast $\mathcal{O}(1)$ eigenvalues.
\end{itemize}
\end{itemize}
\end{lemma}
Note that the case $A_0=2/3$ related to the Turing bifurcation (the dotted vertical line in the bifurcation diagram of Figure~\ref{fig_overview} indicates when $A=2\slash3$ to leading order) will be discussed in the next section, see in particular Lemma~\ref{L:HH}. The degenerate case $A_0=1$ is not discussed further as it is not important for the current paper.
\begin{proof}
The leading order expression \eqref{C0} of $\ueq$ follows directly from~\eqref{eq:equilib} for $A_0 \neq 1$ and upon implementing the constraints of Assumption~\ref{H1}. Thus, there is always one equilibrium, which has
$\ueq = \mathcal{O}(\varepsilon)$. When
$A_0>1$, this is the only equilibrium. For
$A_0<1$, there are two more equilibria with $\ueq
=\pm\sqrt{1-A_0}+\mathcal{O}(\varepsilon)$.

When $\varepsilon$ is sufficiently small and $A_0$ is away from $2/3$,
i.e., $\ueq$ is away from $1/{\sqrt3}$ and $3\ueq^2-1$ is not small,
the eigenvalues of the linearisation about the equilibrium are real or
purely imaginary. By Assumption~\ref{H1} we have $B_0=0=C_0$ and the characteristic
polynomial~\eqref{eq:charpoly} becomes
\[
p(\lambda) = 
\left[(\lambda^2-(3(\ueq^0)^2-1)(\lambda^2-\varepsilon^2) 
+ A_0\varepsilon^2\right](D^2\lambda^2-\varepsilon^2)
+ \mathcal{O}(\varepsilon^3\lambda^2+\varepsilon^5 + \varepsilon \lambda^4)\,.
\]
Thus, there is a hyperbolic pair of slow eigenvalues with
$\lambda_{s,1}^2=\varepsilon^2/D^2+\mathcal{O}(\varepsilon^3)$, and, if
$(\ueq^0)^2\neq
1/3$, then there is a pair of fast eigenvalues with
$\lambda_f^2=3(\ueq^0)^2-1+\mathcal{O}(\varepsilon)$ and a second  pair of slow
eigenvalues with $\lambda_{s,2}^2 =
\varepsilon^2\left(1+
(A_0/(3(\ueq^0)^2-1))\right)+\mathcal{O}(\varepsilon^3)$. The results of the lemma now follow immediately. 
\end{proof}

\subsection{A Turing bifurcation}\label{sec:tur}

In this section, we will show how near-equilibrium spatially periodic
patterns can emerge near $\ueq^2=1/3$ through a Turing bifurcation in case that $C$ and $B$ are small and
$A$ is close to $2/3$, i.e., $A_0=2/3$ and $B_0=0=C_0$.\footnote{
However, since the analysis only relies on $\ueq^2$ being close to
$1/3$,
similar results can be derived for the more general case.} That is, we discuss the case where the eigenvalue configuration of the $\mathcal{O}(1)$ equilibria of
\eqref{e:6slow}/\eqref{e:6fast} changes, see Lemma~\ref{L:new1} .

\begin{lemma}
\label{L:HH}
Let Assumption~\ref{H1} hold.
Then there is some $\varepsilon_0>0$ such that for all $0<\varepsilon<\varepsilon_0$ the system \eqref{e:6slow}/\eqref{e:6fast} undergoes a Turing instability at the curve given by
 \begin{align}\label{HHC}
A = \frac23 + \varepsilon \left(\frac{2\sqrt2}{3\sqrt3} - \left(B_1\pm\sqrt3 C_1\right) \right) +\mathcal{O}(\varepsilon^2).
\end{align}
\end{lemma}
It is known that a Turing instability for reaction-diffusion systems is equivalent to a spatial Hamiltonian-Hopf bifurcation; see~\cite[Lemma 2.11]{Scheel2003}, and hence we show the existence of a Hamiltonian-Hopf bifurcation in \eqref{e:6fast}.\footnote{We shall use Turing instability and spatial Hamiltonian-Hopf bifurcation interchangeably throughout the paper.}
\begin{proof}
As seen in section \ref{s:equilbria}, $(\ueq^0)^2=1/3$ will occur for $A_0=2/3$.
Writing
$\ueq
= \pm \sqrt{1/3} +\varepsilon u_{\rm e,1}
+\mathcal{O}(\varepsilon^2)$ and
$A=2/3 + \varepsilon A_1 +\mathcal{O}(\varepsilon^2)$ and using the Ansatz
$\lambda^2 = \varepsilon\mu^2$, the characteristic
polynomial~\eqref{eq:charpoly} can be written as
\[
  p(\lambda) = \varepsilon^3\left[\left(\mu^2\mp\sqrt3 \,u_{\rm
        e,1}\right)^2-3u_{\rm e,1}^2+\frac23 
\right]\mu^2D^2+ \mathcal{O}(\varepsilon^4). 
\]
For a Hamiltonian-Hopf bifurcation, we need a pair of double purely imaginary
roots of the characteristic polynomial. Hence
$u_{\rm e,1}= \mp \sqrt2/3$ and
$\mu^2 = -\sqrt{2/3} +\mathcal{O}(\varepsilon)$.  On the other hand, we
find from the equilibrium equation~\eqref{eq:equilib} that
$u_{\rm e,1}= \mp \sqrt3/2\left(A_1+B_1\pm\sqrt3 C_1\right)$. One can
also verify that the eigenvalues 
(labelled $\lambda_f$ and $\lambda_{s,2}$ in Lemma~\ref{L:new1})
change from pairs on the imaginary
axis to a quadruple in the complex plane upon varying $A$. 
Thus, for $B_0=0=C_0$ a
Hamiltonian-Hopf bifurcation, and hence also a Turing instability \cite[Lemma 2.11]{Scheel2003}, occurs at the curve given by \eqref{HHC}.
\end{proof}

\begin{figure}
\centerline{\includegraphics[width=.5\textwidth]{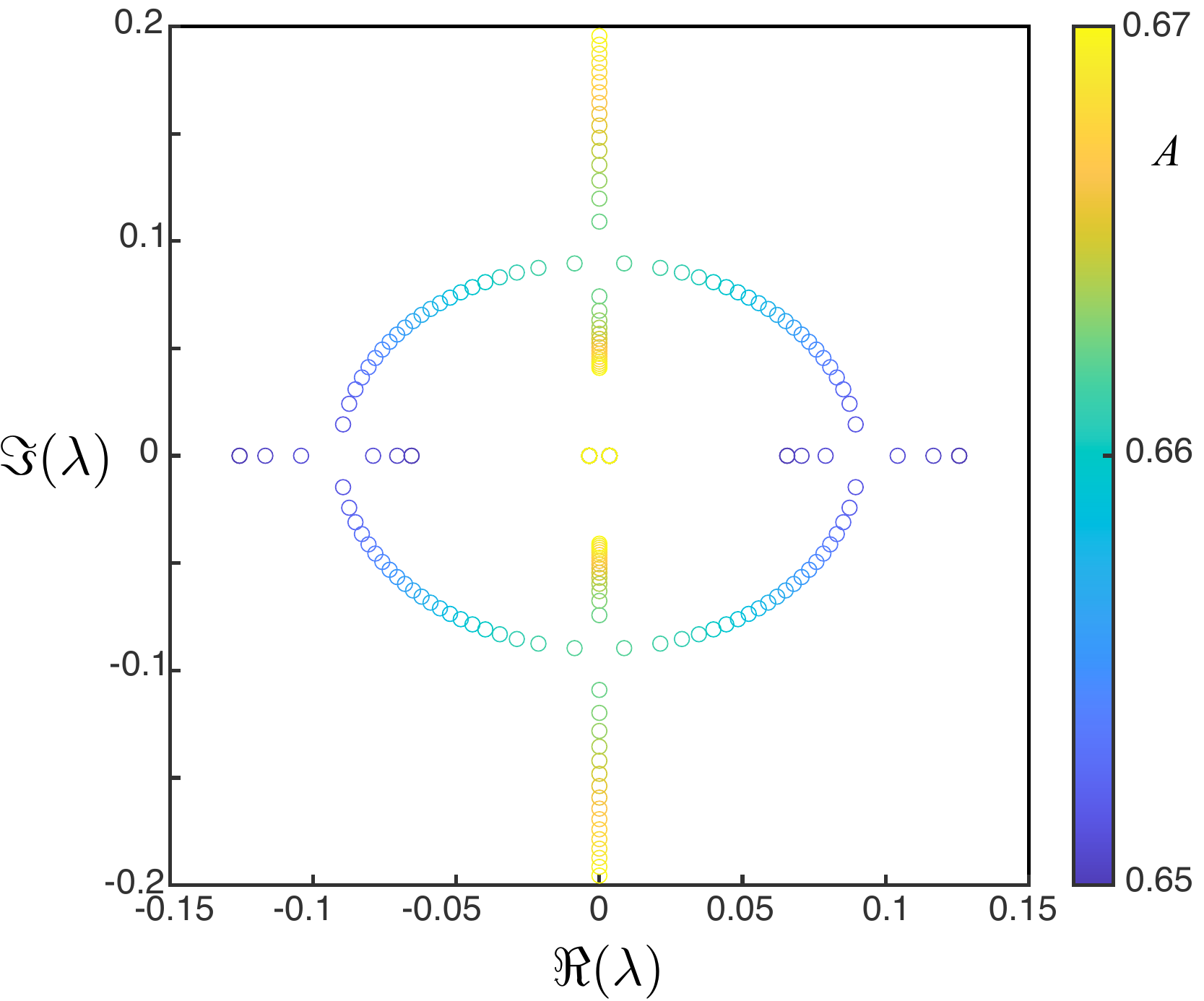}}
\caption{Tracing of the four eigenvalues (labelled $\lambda_f$ and $\lambda_{s,2}$ in Lemma~\ref{L:new1})
involved in the Hamiltonian-Hopf bifurcation for $A$ going
through the Hamiltonian-Hopf curve and $\varepsilon = 0.01$, $B=0.01$,
$C=0$ and $D=3$.}
\label{fig.HH}
\end{figure}
This Hamiltonian-Hopf bifurcation is illustrated in Figure~\ref{fig.HH}, where the eigenvalues
involved in the bifurcation are traced for $A$ going
through the Hamiltonian-Hopf curve \eqref{HHC}.
Note that there is only a very
small window in the $A$ variable for which there is a quadruple of complex 
eigenvalues with nonzero real part, which corresponds to the earlier observation that the
eigenvalues are real or purely imaginary if $A_0\neq 2/3$.

\begin{remark}
In contrast to typical Turing bifurcations we do not observe sinusoidal-like periodic patterns near the bifurcation, see panel \raisebox{.5pt}{\textcircled{\raisebox{-.9pt}
    {1}}} in
Figures~\ref{fig_overview} and \ref{fig_overview22}.
This stems from the fact that at the Hamiltonian-Hopf bifurcation the eigenvalues are $\mathcal{O}(\varepsilon^2)$ and, hence, the period of the bifurcating orbit is expected to be $\mathcal{O}(\varepsilon^{-2})$. So, numerically we do not observe sinusoidal-like periodic patterns bifurcating off for $\varepsilon$ is small. Figure~\ref{fig.HH} corroborates this observation as we only have a  small window in the $A$ variable for which there is a quadruple of complex eigenvalues with nonzero real part. 
\end{remark}

\subsection{The reduced fast system}
Next, we consider the fast dynamics, i.e., the behaviour near the $u$-interfaces in Figures~\ref{fig_overview} and~\ref{fig_overview22}. The reduced fast system is obtained from the singular limit $\varepsilon=0$ of \eqref{e:6fast}. It 
has $v,q,w,r$ constant and the dynamics in $u$ and $p$ is
\begin{eqnarray}
  \label{eq:2fast}
  \begin{array}{lcl}
u_\xi &=& p,\\
 p_\xi &=& -u + u^3 + K(v,w), 
 \end{array}
\end{eqnarray}
where
\begin{eqnarray}
\label{KK}
  K(v,w) := A_0 v + B_0 w + C_0 
\end{eqnarray}
is constant (since $v$ and $w$ are constant).  
This system is Hamiltonian with Hamiltonian
\begin{equation}
  \label{eq.Hfast}
H_f(u,p;K) =
\frac12 p^2 + V_f(u;K), \mbox{ and potential }
V_f(u;K) = -\frac14 u^4+\frac12  u^2 - Ku.
\end{equation}
The equilibria $(\hu_0,\widehat{p}_0)$ in the fast
system~\eqref{eq:2fast} 
are given by $\widehat{p}_0=0$ and the solutions of 
$\hu_0^3-\hu_0+K=0$.
For $|K|\leq 2/(3\sqrt 3)$, there are three $\hu_0$ values
associated with one $K$ value, while there is only $\hu_0$ value
associated with one $K$ value for $|K|>2/(3\sqrt 3)$, see the left panel of
Figure~\ref{fig_uK}. 
\begin{figure}
  \centerline{
  \includegraphics[width=1\textwidth]{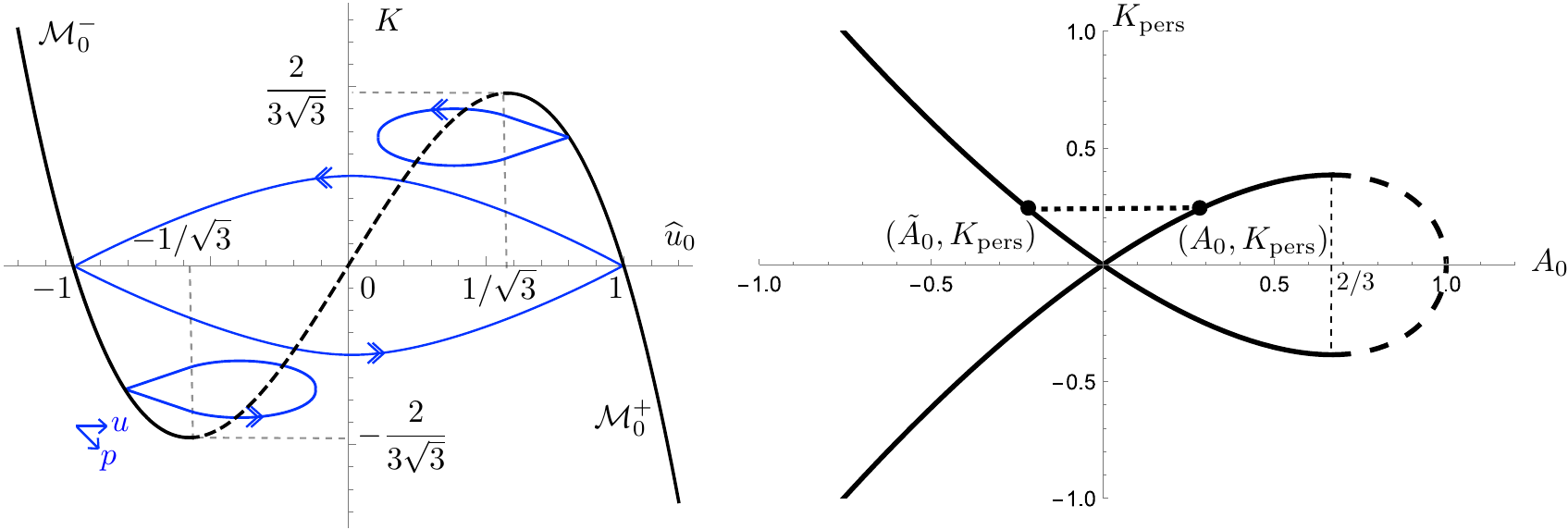}}
  \caption{
 The left panel shows in
   black the relation between $\hu_0$ and $K$ which generates the reduced
   slow manifold $\mathcal{M}_0$, see~\eqref{eq.slowman}. The
equilibria on the left and right branches $\mathcal{M}_0^\pm$ of $\mathcal{M}_0$ are
hyperbolic (solid curves), while the middle ones are elliptic (dashed curve). 
In blue a sketch of the associated reduced fast flow associated to \eqref{eq:2fast}, see also Figure~\ref{fig.fast_dynamics}. For $K=0$ there are two heteroclinic orbits connecting the equilibria on $\mathcal{M}_0^+$ and $\mathcal{M}_0^-$, while for $0<|K| < 2/(3\sqrt3)$ there is one homoclinic orbit connecting to $\mathcal{M}_0^{sgn{(K)}}$.  
The right panel shows the relation $K_{\rm pers}=\pm A_0 \sqrt{1-A_0}$ \eqref{Kpers}, obtained by substituting $\hu_0=\pm\sqrt{1-A_0}$ into \eqref{eq.slowman}. The dashed curve, where $A_0 > 2/3$, is related to the elliptic branch of the reduced slow manifold (shown in the left panel), while the solid curves are related to the hyperbolic branches.
Note that $A_0<0$ on the hyperbolic branches $\mathcal{M}_0^\pm$ when $|\hu_0|>1$.
 }
   \label{fig_uK}
\end{figure}
The potential~$V_f$ and the reduced fast dynamics \eqref{eq:2fast} are sketched in
Figure~\ref{fig.fast_dynamics}.
\begin{figure}
	\centering
	\centerline{\includegraphics{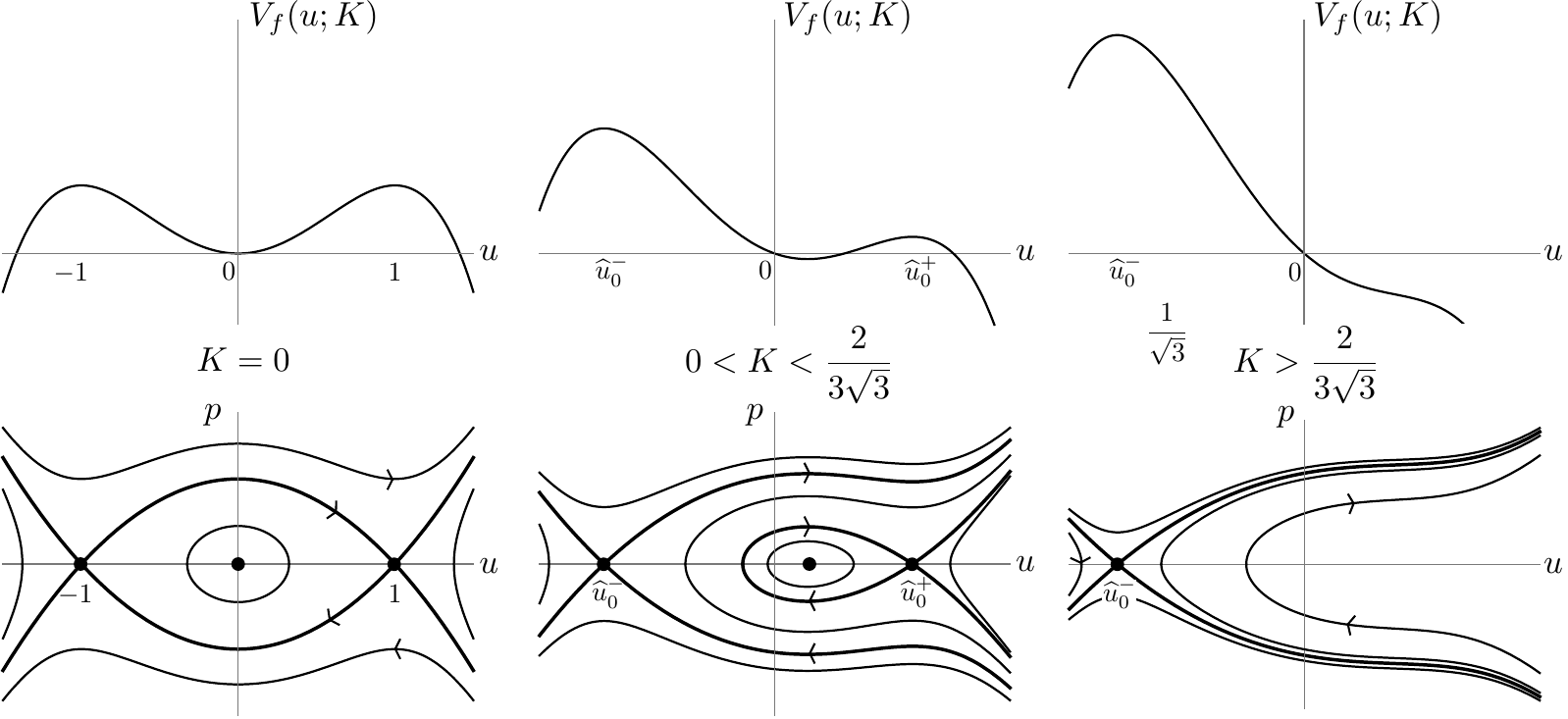}}
        \caption{The potential $V_f(u;K)$ and fast dynamics in the
          various $K$ regions. The figures for $K<0$ follow by the
          symmetry $(u,p,K)\to(-u,-p,-K)$. 
          \label{fig.fast_dynamics}}
\end{figure}

In the full six dimensional reduced fast system (i.e., \eqref{e:6fast} in the
singular limit $\varepsilon=0$) the equilibria form a four dimensional
manifold
\begin{equation}\label{eq.slowman}
  \mathcal{M}_0 = \left\{(\hu_0(v,w),0,v,q,w,r) \:\mid\: 
    \hu_0^3-\hu_0+ K(v,w)=0 ,
    \,\, v,q,w,r \in\mathbb{R}^4\right\}.
\end{equation}
The eigenvalues associated with the linearisation in the reduced fast
system~\eqref{eq:2fast} about the equilibria in $\mathcal{M}_0$ are
given by $\lambda_f(\hu_0) = \pm\sqrt{3\hu_0^2-1}$. Thus, the
equilibria on the left and right branches of $\mathcal{M}_0$ are
hyperbolic, while the middle ones are elliptic. For the analysis later
on, 
the branches with hyperbolic equilibria are of most interest, hence we
define the points $\hu_0^\pm$ as the $u$ value of the equilibria on
$\mathcal{M}_0$ with $\pm\hu_0^\pm>1/\sqrt3$, i.e., $\hu^-_0$
lies on the left branch and $\hu^+_0$ lies on the right branch, see
also 
the left panel of
Figure~\ref{fig_uK}. 
In a
similar way, we define the left and right reduced slow manifolds
$\mathcal{M}_0^-$ respectively $\mathcal{M}_0^+$ as
\[
\mathcal{M}_0^\pm = \left\{(\hu_0(v,w),0,v,q,w,r) \in \mathcal{M}_0\:\mid\: 
\pm\hu_0>\frac{1}{\sqrt3}\right\}.
\]
The hyperbolic four dimensional slow manifolds
$\mathcal{M}_0^\pm$ have five dimensional stable and unstable
manifolds denoted by
$\mathcal{W}_s(\mathcal{M}_0^\pm)$ and
$\mathcal{W}_u(\mathcal{M}_0^\pm)$ respectively. The reduced fast
dynamics creates those stable and unstable manifolds.  The phase
portraits in Figure~\ref{fig.fast_dynamics}, see also the left panel of Figure~\ref{fig_uK}, illustrate that for
$|K|>
2/(3\sqrt3)$ there is only one branch, $\mathcal{M}_0^{-\text{sgn}(K)}$,
  and there are no bounded orbits in the fast
dynamics, hence $\mathcal{W}_s(\mathcal{M}_0^{-\text{sgn}(K)})$ and
$\mathcal{W}_u(\mathcal{M}_0^{-\text{sgn}(K)})$ do not intersect. 
At $K=2/(3\sqrt3)$, we have $\hu_0^+=1/\sqrt3$ and the
hyperbolicity of $\mathcal{M}_0^+$ breaks down ($\mathcal{M}_0^-$ is
still hyperbolic). A similar observation holds for
$K=-2/(3\sqrt3)$ and $\mathcal{M}_0^-$.
For $0<K<2/(3\sqrt3)$, there is one fast homoclinic orbit
associated with $\hu_0^+$ and for $-2/(3\sqrt3)<K < 0$, there is
one fast homoclinic orbit associated with $\hu_0^-$. This implies that parts of the
five dimensional stable and unstable manifolds coincide.
At $K=0$,
$\hu_0^\pm=\pm1$ and there are two fast heteroclinic orbits connecting
these equilibria and hence the two manifolds $\mathcal{M}_0^\pm$.
Thus, parts of the stable
manifold $\mathcal{W}_s(\mathcal{M}_0^\pm)$ coincide with the unstable
manifold $\mathcal{W}_u(\mathcal{M}_0^\mp)$.

As we will show in section~\ref{S:SF} and section~\ref{S:SFS}, $0<|K|<2/(3\sqrt3)$
is related to fast transitions in the profiles of panels \raisebox{.5pt}{\textcircled{\raisebox{-.9pt}
    {1}}} and \raisebox{.5pt}{\textcircled{\raisebox{-.9pt} {2}}} in
Figure~\ref{fig_overview} and panel
\raisebox{.5pt}{\textcircled{\raisebox{-.9pt}
    {1}}} in
Figure~\ref{fig_overview22}, i.e., Theorem \ref{th.one},
while $K=0$ is related to fast transitions
in the profiles in panel \raisebox{.5pt}{\textcircled{\raisebox{-.9pt} {5}}} in
Figure~\ref{fig_overview} and panel \raisebox{.5pt}{\textcircled{\raisebox{-.9pt} {3}}} in
Figure~\ref{fig_overview22}, i.e., Theorem \ref{th.two}, as well as 
panels \raisebox{.5pt}{\textcircled{\raisebox{-.9pt} {6}}} and \raisebox{.5pt}{\textcircled{\raisebox{-.9pt} {7}}} in
Figure~\ref{fig_overview} and panel \raisebox{.5pt}{\textcircled{\raisebox{-.9pt} {4}}} in
Figure~\ref{fig_overview22}, i.e., Theorem \ref{th.two_2}.

In the proof of Theorem \ref{th.one}, the persisting equilibria will play an important role. So here we explicitly discuss the persisting equilibria and the associated reduced fast dynamics under Assumption~\ref{H1} that $B$ and $C$ are small. Under Assumption~\ref{H1}, the persisting equilibria have $\ueq^0=\pm\sqrt{1-A_0}$ in leading order. Substituting $\hu_0=\ueq^0=\pm\sqrt{1-A_0}$ into the relation \eqref{eq.slowman} gives
\begin{align}
\label{Kpers}
K_{\rm pers}=\pm A_0\sqrt{1-A_0}.
\end{align}
This relation is depicted in the right panel of Figure~\ref{fig_uK}. 
In the reduced fast dynamics, if $0<A_0<2/3$, these equilibria have a homoclinic orbit associated to them, see the middle panels of Figure~\ref{fig.fast_dynamics}. 
In these phase portraits with $K=K_{\rm pers}$, there is
another hyperbolic equilibrium with no bounded connections. As this equilibrium has the same $K$-value, it is thus related to another $A_0$-value. In particular, it is related to $\ueq^0=\mp\sqrt{1-\tilde A_0}$, with $-1/3 <\tilde A_0<0$ and $\tilde A_0$ determined by
$-\tilde A_0 \sqrt{1-\tilde A_0}=K_{\rm pers}$, see also the right panel of Figure~\ref{fig_uK}.
If $A_0=0$, i.e., $K=0$, then
there are heteroclinic connections between $\ueq^{0}=+1$ and
$\ueq^{0}=-1$, see the left panels of Figure~\ref{fig.fast_dynamics}.
If $A_0<-1/3$, then the equilibria $\ueq^0=\mp\sqrt{1-A_0}$
do not have homoclinic or heteroclinic connections connected to them in
the reduced fast system, see the right panels of
Figure~\ref{fig.fast_dynamics}. 
These observations are summarised below, see also Figure~\ref{fig_uK}.
\begin{lemma}
\label{l:fd}
Let Assumption~\ref{H1} hold. Then the persisting equilibria satisfy $\ueq^0=\pm\sqrt{1-A_0}$ to leading order. In the reduced fast system, these equilibria have 
\begin{itemize}
\item a homoclinic orbit 
 for $0 < A_0<2/3$;
\item two heteroclinic orbits connecting $\ueq^0=\pm 1$ to $\ueq^0 = \mp 1$ for $A_0=0$; and
\item no heteroclinic or homoclinic orbits for $A_0<0$ and $A_0>2/3$. 
\end{itemize} 
\end{lemma}

\begin{remark}
Assumption~\ref{H1} simplifies the relation~\eqref{KK} to $K(v,w)=A_0v$. The persisting equilibria have $v=\ueq$ in leading order, hence this relation becomes $A_0 v = \pm A_0\sqrt{1-A_0}=K_{\rm pers}$~\eqref{Kpers}. Hence, this is consistent with the derivation of $K_{\rm pers}$ via relation~\eqref{eq.slowman}.
\end{remark}

\subsection{The reduced slow system}
Next, we consider the slow dynamics, i.e., the behaviour away from the $u$-interfaces in Figures~\ref{fig_overview} and~\ref{fig_overview22}. 
The reduced slow system dynamics is obtained from the singular limit $\varepsilon=0$ of \eqref{e:6slow}. It lies on the
manifold $\mathcal{M}_0$ given by~\eqref{eq.slowman}. 
The dynamics on the slow manifold is defined by the system
\begin{equation}\label{e:rslow}
\begin{array}{rcl}
v_x &=& q\,,\\
q_x &=&v - \hu_0(v,w)\,, \\
w_x &=& \dfrac{r}{D}\,,\\[2mm]
r_x &=& \dfrac{1}{D}(w - \hu_0(v,w))\,. \\
\end{array}
\end{equation}
Hence, equilibria for the reduced slow system are limits for
$\varepsilon\to0$ of the equilibria determined
in~\eqref{eq:equilib}. Note that away from the equilibria, $K(v,w)$ is
not constant (unless all parameters are order $\varepsilon$, which
implies $K=0$, irrespective of the values of $v$ and $w$) and
therefore, $u=\hu_0(v,w)$ will vary along with the dynamics of the
slow manifold when $K\neq 0$.
We will study the reduced slow system \eqref{e:rslow} in more detail in the upcoming sections for the different types of periodic patterns.

\subsection{Slow-fast periodic
  solutions and persisting locally invariant manifolds}\label{S:per_man}

For any $\delta>0$, define the truncated slow manifolds
\begin{align}
\label{SM}
\mathcal{M}_{0,\delta}^\pm = \left\{(\hu_0(v,w),0,v,q,w,r) \in \mathcal{M}_0^\pm\:\mid\: 
\pm\hu_0\geq\frac{1}{\sqrt3}+\delta\right\}.
\end{align}
By normal hyperbolicity, these manifolds, and their stable and
unstable manifolds, persist (for fixed $\delta$) as locally invariant slow
manifolds $\mathcal{M}_{\varepsilon,\delta}^\pm$ with associated stable
and unstable manifolds in the full dynamics for $\varepsilon \neq 0$
\cite{F79,J95,K99}.  The coinciding stable and unstable manifolds
associated with the homoclinic orbit will persist as the full system
is still Hamiltonian and hence this manifold is the levelset of
the conserved Hamiltonian, which is smoothly changed. 
However, while the unperturbed truncated slow manifolds $\mathcal{M}_{0,\delta}^\pm$ consist of fast equilibria of the reduced fast system~(\ref{eq:2fast}), only the true equilibria (\ref{eq:equilib}) persist in the full perturbed system (\ref{e:6slow})/(\ref{e:6fast})
and the existence of a homoclinic
  orbit to the persisting equilibria
  requires some further analysis and is shown 
  in sections~\ref{S:SF} and~\ref{S:SFS}.
In contrast, the coinciding of the stable and unstable
manifolds associated with the heteroclinic orbit will generically not
persist and further study is thus also needed to obtain conditions that give a
persisting heteroclinic orbit, see section~\ref{S:SFS}.

As seen in Lemma~\ref{L:new1}, when $B$ and $C$ are small, one of the equilibria satisfies
$\ueq^\varepsilon = \mathcal{O}(\varepsilon)$ and, when $A_0<1$, there are
also equilibria with
$\ueq^\varepsilon = \pm\sqrt{1-A_0}+\mathcal{O}(\varepsilon)$. When
$A_0<2/3$, the last two equilibria are hyperbolic and hence part of
the persisting locally invariant slow manifolds $\mathcal{M}_{\varepsilon,\delta}^\pm$.
These equilibria have three dimensional stable and unstable manifolds
$\mathcal{W}_{u,s}(\ueq^\varepsilon)$ and these manifolds are embedded in the
stable respectively unstable manifolds of $\mathcal{M}_{\varepsilon,\delta}^\pm$.
We will use the persisting locally invariant slow manifolds and the slow-fast dynamics to
investigate the persistence of the intersection of their stable and
unstable manifolds and to prove the existence of slow-fast periodic
solutions in the full system and give approximations for these
solutions.

A slow-fast periodic solution has transitions between
slow and fast behaviour. The minimal number of transitions will be
one: there will be one slow phase near
$\mathcal{M}_{\varepsilon,\delta}^+$ or
$\mathcal{M}_{\varepsilon,\delta}^-$ and one fast phase.
The plots marked with \raisebox{.5pt}{\textcircled{\raisebox{-.9pt}
    {1}}} and \raisebox{.5pt}{\textcircled{\raisebox{-.9pt} {2}}} in
Figure~\ref{fig_overview} and
\raisebox{.5pt}{\textcircled{\raisebox{-.9pt}
    {1}}}
    in
Figure~\ref{fig_overview22} correspond to such type of solutions. These
solutions require a homoclinic connection, 
so they can only occur when $0<|K(v,w)|<2/(3\sqrt3)$ in lowest order
during the fast phase, see the middle panels of
Figure~\ref{fig.fast_dynamics}.

To analyse such solutions and capture the spatial dynamics, we will
look for $2L$-periodic solutions with $x\in[-L,L]$ and divide this interval
in one slow and one fast region. 
We write 
\[
  [-L,L]   
  = \left[-L,-\sqrt{\varepsilon}\right]\cup
\left(-\sqrt{\varepsilon},\sqrt{\varepsilon}\right)\cup\left[\sqrt{\varepsilon},L\right] 
=: I_s\cup I_f, \text{ with } I_f:=\left(-\sqrt{\varepsilon},\sqrt{\varepsilon}\right),
\]
where we, without loss of generality, centred the fast
transition at $x=0$, see also the left plot in
Figure~\ref{f:periodic_setup}.
The choice of the asymptotic width of the fast interval to be $2\sqrt\varepsilon$ is
arbitrary and not intrinsically related to the original problem, but rather a necessary ingredient of the geometric approach. Actually, any other choice $M \varepsilon^\chi$ with $\chi \in (0,1)$, such that the fast interval vanishes in the singular limit $\varepsilon \to 0$ in the slow scaling, but blows up to the whole real line in the fast scaling, will work. Note that the asymptotic scaling also does not play an essential role in
the description of the solution. In the slow regions, the dynamics
will take place near one of the two slow locally invariant
manifolds~$\mathcal{M}_{\varepsilon,\delta}^\pm$. In the fast region, the slow
variables $v$ and $w$ are constant in lowest order and the fast variable $u$ leaves
the slow manifold, but has to return to the same manifold as there is
only one transition. This return has to correspond to the dynamics
staying close to a homoclinic connection to the slow manifold. 
\begin{figure}
	\centering
		{\includegraphics{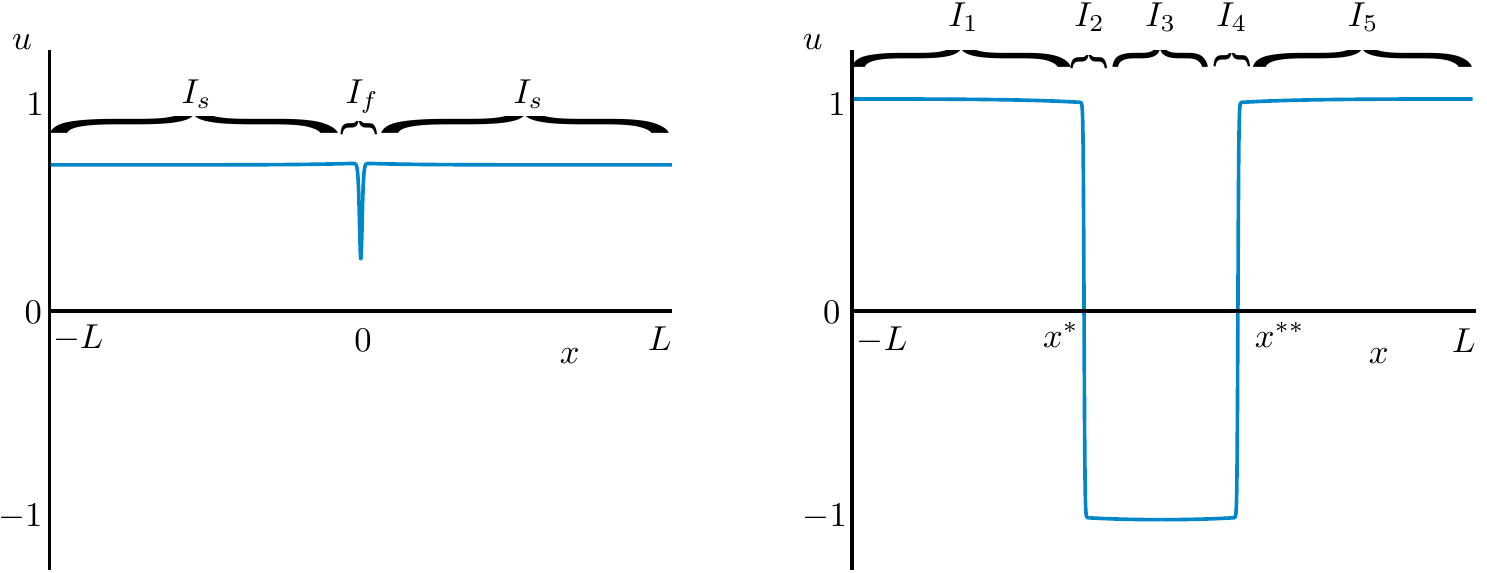}}
	\caption{Periodic existence setup. The left panel is associated to periodic solutions with one fast transition, while the right panel is associated to periodic solutions with two fast transitions.}         
            \label{f:periodic_setup}
\end{figure}

A different type of periodic solution is obtained when there are two
transitions and both fast phases involve a heteroclinic
connection of the stable and unstable manifolds, see, for instance,
the plots marked with \raisebox{.5pt}{\textcircled{\raisebox{-.9pt}
    {4}}}-\raisebox{.5pt}{\textcircled{\raisebox{-.9pt} {7}}} in
Figure~\ref{fig_overview} and
\raisebox{.5pt}{\textcircled{\raisebox{-.9pt} {3}}}-\raisebox{.5pt}{\textcircled{\raisebox{-.9pt} {5}}} in
Figure~\ref{fig_overview22}.  These fast transitions should occur near
$K(v,w)=0$ as this is the only value for which heteroclinic orbits
exist, see Figure~\ref{fig.fast_dynamics}. Note
that $K(v,w)$ can become small (order $\varepsilon$) in two distinctive
ways. Firstly, the system parameters $A$, $B$, and $C$ can be small
(order $\varepsilon$) as is the case for
\raisebox{.5pt}{\textcircled{\raisebox{-.9pt} {4}}} and \raisebox{.5pt}{\textcircled{\raisebox{-.9pt} {5}}} in
Figure~\ref{fig_overview}
and for
\raisebox{.5pt}{\textcircled{\raisebox{-.9pt} {3}}} and \raisebox{.5pt}{\textcircled{\raisebox{-.9pt} {5}}} in
Figure~\ref{fig_overview22}. 
Alternatively, $B$, $C$, $v$ and $w$ can be small (order
$\varepsilon$) while $A$ is
order 1 near the fast transition as is the case for
\raisebox{.5pt}{\textcircled{\raisebox{-.9pt} {6}}} and \raisebox{.5pt}{\textcircled{\raisebox{-.9pt} {7}}} in
Figure~\ref{fig_overview} and for \raisebox{.5pt}{\textcircled{\raisebox{-.9pt} {4}}} in
Figure~\ref{fig_overview22}.

To analyse such $2L$-periodic solutions, we write
\begin{align} \label{intervals}
\begin{aligned}
  [-L,L] &= \left[-L,x^*-\sqrt{\varepsilon}\right]\cup
\left(x^*-\sqrt{\varepsilon},x^*+\sqrt{\varepsilon}\right)\cup\left[x^*+\sqrt{\varepsilon},x^{**}-\sqrt{\varepsilon}\right]  
\\ & \qquad\cup\left(x^{**}-\sqrt{\varepsilon},x^{**}+\sqrt{\varepsilon}\right)   
\cup\left[x^{**}+\sqrt{\varepsilon},L\right]  
\\ &=:I_1\cup I_2 \cup I_3\cup I_4 \cup I_5,
\end{aligned}
\end{align}
where the large odd numbered intervals are expected to be dominated by slow dynamics and
the small even numbered ones by fast dynamics, see the right plot in
Figure~\ref{f:periodic_setup}, and where $-L < x^* < x^{**} < L$ need to be determined. Again, in the slow regions, the
dynamics will take place near one of the two slow locally invariant
manifolds~$\mathcal{M}_{\varepsilon,\delta}^\pm$. The fast dynamics uses
a heteroclinic connection between these two manifolds.

Other periodic slow-fast solutions in Figures~\ref{fig_overview} and \ref{fig_overview22} can be
obtained by combining these two scenarios. For instance, the plot
marked with \raisebox{.5pt}{\textcircled{\raisebox{-.9pt} {7}}} in
Figure~\ref{fig_overview22} has three transitions: the two outer
ones involve the heteroclinic connections (happing near $K(v,w)=0$ as
both $v, w$ are small), while the middle one involves the homoclinic
connection (as $0<|K(v,w)|<2/(3\sqrt3)$ here).

\section{Proof of Theorem~\ref{th.one}: a slow-fast periodic solution with one fast transition}
  \label{S:SF}
In this section, we study slow-fast periodic solutions with one fast transition and, in particular, prove Theorem~\ref{th.one}.
  As we have discussed in \cref{S:SETUP}, a slow-fast periodic solution with one
  fast transition has one slow phase near $\mathcal{M}^+_{\varepsilon,\delta}$ (or
  $\mathcal{M}^-_{\varepsilon,\delta}$) and one fast phase. The fast phase takes place
  near a homoclinic connection between the stable and unstable
  manifolds of $\mathcal{M}_{\varepsilon,\delta}^+$ (or $\mathcal{M}_{\varepsilon,\delta}^-$)
  as it has to return to the slow phase on the same branch of the slow
  manifold. Such connections can only occur when
  $0<|K(v,w)|<2/(3\sqrt3)$ in lowest order during the fast
  phase, see the middle panels of Figure~\ref{fig.fast_dynamics}.
  We assumed in Theorem~\ref{th.one} that both
  $B$ and $C$ are small, i.e. $B_0=0=C_0.$\footnote{This assumption is not without loss of generality, but we postulate that similar results can be obtained in the more general case. }
Then, $K(v,w)$ \eqref{KK} simplifies to $K=A_0v$, with, by assumption, $A_0 \neq 0$. This implies that
  we can use the
  variable~$v$, instead of $K(v,w)$, to characterize the relevant parts of
  the two branches of the slow manifold at which transitions to the
  fast phase can occur:
\[
\mathcal{M}_{0,\rm 1fs}^\pm= \left\{(\hu^\pm_0(v),0,v,q,w,r) \:\mid\: 
  (\hu^\pm_0)^3-\hu^\pm_0+A_0v=0; \,\,
v,q,w,r \in\mathbb{R}^4, \,\frac1{\sqrt3}<\pm\hu_0^\pm(v)<1 \right \}.
\]
Here we used that if $K=A_0v>0$ then $\hu_0^+(v)$ has a homoclinic
connection in the fast zeroth order dynamics, while the same holds for
$\hu_0^-(v)$ if  $K=A_0v<0$, see Figure~\ref{fig.fast_dynamics} and Lemma~\ref{l:fd}.

We start the search for periodic orbits with one fast transition with
a heuristic investigation of such solutions.
Assume that $\bu_s(x,\varepsilon) = (u_s, p_s, v_s, q_s, w_s, r_s)(x,\varepsilon)$ is a $2L$-periodic solution with one fast
transition, represented in the slow coordinates. In the fast
coordinates this solution can be written as
\[
\bu_f(\xi,\varepsilon) := \bu_s(\varepsilon\xi,\varepsilon),\quad \xi\in[-T,T], \quad
\mbox{with} \quad T=\frac L {\varepsilon}.
\]
We will first show that the transition to the fast region has to occur
near a fixed point of the full system, i.e., near
$(\hu^\pm_0(v),0,v,q,w,r) =\pm\sqrt{1-A_0}(1,0,1,0,1,0)$.  During the
slow phase $I_s$, 
the fast variables are near the slow manifold:
\[
  u_s(x) = \hu_0(v(x)) + o(1),
  \quad p_s(x) = 0 + o(1),
  \quad x\in I_s, 
\]
where we recall that we assumed $B$ and $C$ are small and hence $\hu_0$ does not explicitly depend on $w(x)$.
The slow flow is to leading order determined by \eqref{e:rslow}.
During the fast phase~$I_f$, the slow variables are constant in lowest
order:
\[
v_f(\xi) = v_0+o(1); \quad
q_f (\xi) = q_0+o(1); \quad
w_f (\xi) = w_0+o(1); \quad
r_f (\xi) = r_0+o(1); \quad \varepsilon \xi \in I_f;
\]
and $(u_f,p_f)$ move fast near a homoclinic connection to one of the
points $(\hu_0^\pm(v_0),0,v_0,q_0,w_0,\\r_0)$ on the slow manifold
$\mathcal{M}^\pm_{0,\rm 1fs}$. Hence
\[
u_f(\xi) = u_h(\xi, v_0) + o(1); \quad p_f(\xi) = p_h(\xi, v_0) + o(1), 
\]
where
$(u_h,p_h)(\xi,v_0)$ are the homoclinic connection to
$(\hu^\pm_0(v_0),0)$ in the reduced fast system.
To determine near which of the hyperbolic
fixed points on the slow manifold the fast transition takes place,
we evaluate the change in the slow variables during the fast
phase.  Using the fast equations~\eqref{e:6fast}, the change in $q$ is given
by
\begin{eqnarray}
  \Delta^f_q(\varepsilon) &:=& q_f(1/\sqrt\varepsilon) - q_f(-1/\sqrt\varepsilon) = 
                               \int_{-\frac1{\sqrt\varepsilon}}^{\frac1{\sqrt\varepsilon}}
                               \frac{dq_f}{d\xi}\,d\xi =  
                               \varepsilon\int_{-\frac1{\sqrt\varepsilon}}^{\frac1{\sqrt\varepsilon}}
                               (v_f-u_f)\,d\xi\nonumber \\ \label{eq:delta_fq}
                          &=&\varepsilon\int_{-\frac1{\sqrt\varepsilon}}^{\frac1{\sqrt\varepsilon}}
                              (v_0-u_h(\xi,v_0)+o(1))\,d\xi\\ \nonumber
                          &=& 2(v_0-\hu^\pm_0(v_0))\sqrt\varepsilon +
                             o(\sqrt\varepsilon),
\end{eqnarray}
since $u_h(\xi,v_0)$ converges to $\hu^\pm_0(v_0)$ and 
$\varepsilon\int_{-\infty}^\infty (\hu^\pm_0(v_0)-u_h(\xi,v_0))\,d\xi$ is bounded. 
So, in the singular limit $\varepsilon \to 0$,
the slow solution $q_s(x)$ is continuous at $x=0$. Similar arguments
for the slow variables $r$, $v$ and $w$  show that these slow
solutions are also continuous at $x=0$ in lowest order.
By \eqref{e:rslow}, this implies that the lowest order slow solutions
are constant with $v_0=\hu^\pm_0(v_0)=w_0$, $q_0=0=r_0$. The
  only equilibria that satisfy this relation and the constraint $1/\sqrt3<|\hu^\pm_0|<1$ are
  $\hu_0^\pm=\pm\sqrt{1-A_0}$, with $0<A_0<2/3$.
  Thus, if a $2L$-periodic solution with one transition exists, then
  necessarily $0<A_0<2/3$ and its zeroth order approximation is
\begin{eqnarray*}  
v^0(x)=w^0(x)&=&v_0=w_0=\pm\sqrt{1-A_0}; \quad q^0(x)=r^0(x)=0; \quad
x\in[-L,L];\\
(u^0_s(x),p^0_s(x)) &=&\left( \pm\sqrt{1-A_0},0\right),\,\, x\in I_s;\\
  (u^0_f(\xi),p_f^0(\xi)) &=& \pm
(u_h(\xi,\pm\sqrt{1-A_0}),p_h(\xi,\pm\sqrt{1-A_0})),\,\, \varepsilon \xi \in I_f.
\end{eqnarray*}
 Recall that $(u_h(\xi,v_0),p_h(\xi,v_0))$ is the homoclinic solution of
\eqref{eq:2fast} with $B_0=0=C_0$ (i.e., $K=A_0v_0$). Using that the Hamiltonian of the fast
system~\eqref{eq:2fast} is conserved, we find that the extremal point
of the $u$-coordinate of this homoclinic orbit is given by
\begin{align}
\label{U_EXT}
u_0^{\rm ext} =\pm\left(\sqrt{2A_0}-\sqrt{1-A_0}\right).
\end{align}
This extremal point goes to $\mp 1$ for $A_0\to 0$,
illustrating the fact that the fast homoclinic orbit becomes
heteroclinic when $A_0 \to 0$. For $A_0\to 2/3$ the extremal
point converges to its fixed point value $\pm 1/\sqrt3$, illustrating that the
fast homoclinic orbit degenerates in a Hamiltonian-Hopf bifurcation of
the fixed points. See also panels
\raisebox{.5pt}{\textcircled{\raisebox{-.9pt}
    {1}}}-\raisebox{.5pt}{\textcircled{\raisebox{-.9pt}
    {3}}} of
Figure~\ref{fig_overview} and panels
\raisebox{.5pt}{\textcircled{\raisebox{-.9pt}
    {1}}} and \raisebox{.5pt}{\textcircled{\raisebox{-.9pt}
    {2}}} of Figure~\ref{fig_overview22}.

After these heuristics, we now show that
the orbit homoclinic to
$\bue^0=\pm\sqrt{1-A_0}(1,0,1,0,\\1,0)$ for
$0<A_0<2/3$ persists as a homoclinic orbit to the persisting fixed
point $\bue^\varepsilon=\bue = \bue^0 +
\mathcal{O}(\varepsilon)$ (see~\eqref{eq:equilib}) for
$\varepsilon$ small and can be used to construct a slow-fast
periodic solution.

First we observe that the five dimensional stable and unstable manifolds
$\mathcal{W}_{u,s}(\mathcal{M}_{0,\delta}^{\pm})$ transversely intersect the hyperplane
\[
  \mathcal{P} = \{(u,0,v,0,w,0)\mid u,v,w \in \mathbb{R}\}
\]
with the two dimensional intersection given by 
\[
\mathcal{W}_{u,s}(\mathcal{M}_{0,\delta}^{\pm}) \cap
  \mathcal{P} = \{(u_h(0,v),0,v,0,w,0)\mid v,w \in \mathbb{R}\}.
\]
Recall that $u_h(\xi,v)$ is the $u$-component of the symmetric
homoclinic connection to $(\hu_0^\pm(v),0)$ in the reduced fast
system.  Thus for $\varepsilon$ small, the stable and unstable manifolds of
the persisting manifolds $\mathcal{W}_{u,s}(\mathcal{M}_{\varepsilon,\delta}^{\pm})$
will also intersect $\mathcal P$  and the intersection will be nearby 
$\mathcal{W}_{u,s}(\mathcal{M}_{0,\delta}^{\pm})\cap
  \mathcal{P}$.

  A persisting fast homoclinic orbit will be in the intersection of
  the stable and unstable manifolds of the persisting fixed
  point~$\bue^\varepsilon$. The three dimensional stable
  manifold~$\mathcal{W}_s(\bue^\varepsilon)$ lies in the
  five dimensional stable manifold
  $\mathcal{W}_{s}(\mathcal{M}_{\varepsilon,\delta}^{\pm})$ and the homoclinic
  orbit
  \[
    (u_h(\xi,\sqrt{1-A_0}), p_h(\xi,\sqrt{1-A_0}),
    \sqrt{1-A_0},0,\sqrt{1-A_0},0), \quad \xi\geq 0,
  \]
lies nearby $\mathcal{W}_{s}(\mathcal{M}_{\varepsilon,\delta}^{\pm})$. So a
dimension count gives that there has to be at least one point in which
$\mathcal{W}_{s}(\mathcal{M}_{\varepsilon,\delta}^{\pm})$ intersects
$\mathcal{P}$.

Thus there exists an orbit
$\bu^\varepsilon_s(\xi)\in \mathcal{W}_{s}
(\bue^\varepsilon)$ which intersects $\mathcal P$ in
$\xi=0$.  The Hamiltonian nature of the equations gives a
reversibility symmetry in the system: if $\bu(\xi)$ is a solution,
then $(u,-p,v,-q,w,-r)(-\xi)$ is a solution too.  This implies that
$\bu^\varepsilon_u(\xi):=
(u^\varepsilon_s,-p^\varepsilon_s,v^\varepsilon_s,-q^\varepsilon_s,
w^\varepsilon_s,-r^\varepsilon_s)(\xi)$
is a solution on the unstable manifold
$\mathcal{W}_{u}(\bue^\varepsilon)$ which intersects $\mathcal{P}$ at
$\xi=0$ in the same points as $\bu^\varepsilon_s$. In other words,
they form a homoclinic connection to $\bue^\varepsilon$ nearby the
homoclinic connection to $\bue^0$.

Now we have shown the persistence of a homoclinic orbit to the
persisting fixed point $\bu_e^\varepsilon$, we can use Fenichel's
singular perturbation theory, see for instance \cite{F79,J95,K99}, to justify the existence of a slow-fast
periodic orbit and to finalise the proof of Theorem~\ref{th.one}. We omit the further technical details.

\begin{remark}
The results of Theorem~\ref{th.one} are independent of the period of
the $2L$-periodic solution (though $\varepsilon_0$ will depend on $L$
with $\varepsilon_0$ decreasing once $L$ gets large or small as can be
seen in Appendix~\ref{A:HOT}).
We need
to compute the next order correction terms of the periodic solutions
to see how the period comes into play. This computation can be found in Appendix~\ref{A:HOT}.
\end{remark}

\section{Slow-fast periodic solutions with two fast transitions}
\label{S:SFS}

Theorem~\ref{th.one} is valid for $0<A_0<2/3$. For $A_0$ near
$2/3$ a Hamiltonian-Hopf bifurcation occurs, see Lemma~\ref{L:HH}, and we see the creation of the near-equilibrium periodic pattern. 
For $A_0$ near $0$, the homoclinic
orbit associated with the persisting fixed points is near the transition to a pair of heteroclinic orbits, see Lemma~\ref{l:fd}. The
extremal point $u_0^{\rm ext}$ \eqref{U_EXT} becomes $\mp 1 + \mathcal{O}(\sqrt{A_0})$ and
the passage time near the extremal point is of the order
$\log (A_0)$ as follows from the linearisation, hence becomes slow. In this case for $A_0$ near $0$,
a transition to orbits with two fast transitions takes place.

From Figure~\ref{fig.fast_dynamics} and Lemma~\ref{l:fd} it follows that, {\emph{a priori}}, there can be two types of periodic solutions with two fast
transitions: one is a solution with two jumps involving solutions near
the heteroclinic orbits and going from near $\mathcal{M}_\varepsilon^-$
to near $\mathcal{M}_\varepsilon^+$ and back. The other is a solution
near two homoclinics. The latter solution will involve only
$\mathcal{M}_\varepsilon^-$ or only $\mathcal{M}_\varepsilon^+$ and the
continuity condition from the previous section will need to be
satisfied at both jump points. This leads to solutions similar to the
ones of the previous section and we will not further study these type of solutions.

So in this section we focus on periodic solutions with two fast
transitions formed by two heteroclinic connections. If we assume that Assumption~\ref{H1} holds such that $B_0=0=C_0$, then this means that 
during the fast phase, the system should satisfy
$0=K(v_0,w_0)=A_0 v_0$, see Lemma~\ref{l:fd}. Thus, either
$A_0=0$ or $v_0=0$ during the fast transition and the fast variable
$u$ will change from near $-1$ to near $+1$ or the other way around.
As described in section~\ref{S:per_man} and sketched in the right plot
in Figure~\ref{f:periodic_setup}, to characterize a slow-fast periodic
solution with two fast transitions, we define $2L$ to be the period
(in the slow variables)
and split the interval $[-L,L]$ in five sub-intervals:
$[-L,L] = I_1\cup I_2 \cup I_3\cup I_4 \cup I_5$, where the odd
numbered intervals are dominated by slow dynamics and the even
numbered ones by fast dynamics.  Using the translation invariance, we
can assume that
the periodic pattern starts near 
$\mathcal{M}_\varepsilon^+$ at $x=-L$
with $q(-L)=0$, has a transition in $I_2$ to near
$\mathcal{M}_\varepsilon^-$
with $u$ changing fast from near 
$+1$ to near $-1$, 
continues near 
$\mathcal{M}_\varepsilon^-$
 in $I_3$, changes back
to near 
$\mathcal{M}_\varepsilon^+$ in $I_4$ with $u$ changing fast from
near 
$-1$ to near $+1$, 
and continues near 
$\mathcal{M}_\varepsilon^+$ in
$I_5$ to $q(L)=0$, see the right plot in Figure~\ref{f:periodic_setup}. 
We write the interval $I_2$
as centered around a point~$x^*$ that will be determined later, i.e.,
$I_2 = \left(x^*-\sqrt{\varepsilon}, x^*+{\sqrt\varepsilon}\right)$.  Again,
the choice of width of order ${\sqrt\varepsilon}$ in the slow coordinates
is not essential. Similarly, $I_4$ has the same width and is centered
around~$x^{**}$.

The condition $0=K(v_0,w_0)=A_0v_0$ during the fast phase implies that
either $A_0=0$ or $v_0=0$ during the fast phase. In
section~\ref{S:SFS_Asmall}, we will consider the case $A_0=0$, i.e.,
$A$ small, and in section~\ref{S:SFS_Alarge}, we will consider the
case $A_0\neq0$, i.e., $v_0=0$ during the fast jump. That is, in section~\ref{S:SFS_Asmall} we prove Theorem~\ref{th.two} and the proof of Theorem~\ref{th.two_2} is discussed in section~\ref{S:SFS_Alarge}.

\subsection{Proof of Theorem~\ref{th.two}: all parameters small such that $A_0=0$}
\label{S:SFS_Asmall}
First we focus on the case where, in addition to $B_0=0=C_0$ (by Assumption~\ref{H1}),
$A_0=0$. 
For $A_0=0$, the expression $K(v,w)=A_0v=0$ holds for
all $v$ and $w$ and $\hu_0^\pm(v)=\pm1$. Thus at $A_0=0$, the
hyperbolic parts of the slow manifold $\mathcal{M}_0$ are uniform in
the slow variables and are now given by
\begin{equation}\label{eq.slowman55}
  \mathcal{M}^\pm_0 = \left\{\, 
    (\pm 1, 0, v,q,w,r)\mid v,w,q,r \in \mathbb{R}\,\right\}.
\end{equation}
During the fast phases $I_2$ and $I_4$, the slow variables are
constant in lowest order. We denote the lowest order approximation of
the slow variables by $v_*,q_*,w_*,r_*$ in $I_2$ and
$v_{**},q_{**},w_{**},r_{**}$ respectively in $I_4$.  The reduced fast
system~\eqref{eq:2fast} in the $(u,p)$ variables becomes
\begin{equation}\label{eq:fast0}
\begin{array}{rcl}
u_\xi &=& p,\\
  p_\xi &=& -u + u^3.
\end{array}
\end{equation}
The $(u,p)$-phase plane is depicted in Figure~\ref{f:reduce_planes}(a), see also the bottom left panel of Figure~\ref{fig.fast_dynamics},
and the heteroclinic connections are known explicitly and given by
\begin{equation}\label{FAST}
u_{0}^\pm(\xi) = \pm\tanh\left(\frac{\xi}{\sqrt{2}} \right)\quad \mbox{and}\quad
p_0^\pm(\xi) = \pm\frac{1}{\sqrt{2}}\sech^2\left(\frac{\xi}{\sqrt{2}} \right).
\end{equation}
\begin{figure}
	\centering
	\includegraphics[width=0.9\linewidth]{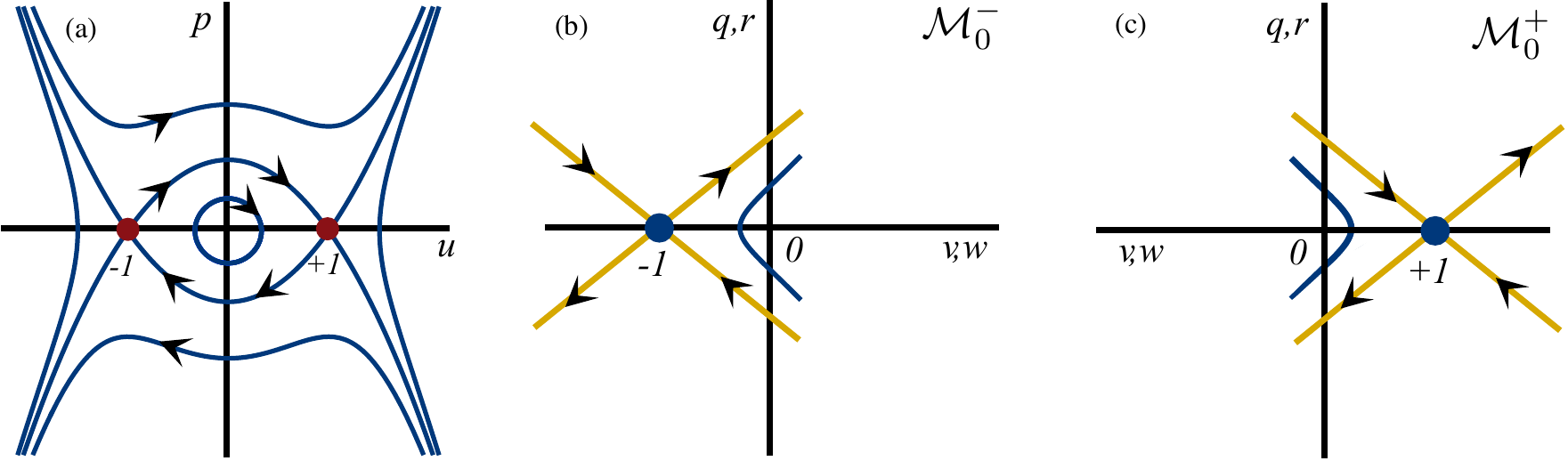}
	\caption{The reduced fast and slow system phase planes for
          $A_0=0$: the reduced fast dynamics is sketched in panel (a);
          panels (b) and (c) show the linear slow dynamics on
          $\mathcal{M}^-_0$ respectively on
          $\mathcal{M}_0^+$. \label{f:reduce_planes}}
\end{figure}
Thus in the fast dynamics on $I_2$ we have in lowest order 
\begin{equation}\label{eq.fast_I2_0}
\begin{array}{{l}}
  u_f(\xi) = u_{0}^-(\xi)+ o(1)
  , \quad
  p_f(\xi) = p_{0}^-(\xi)+ o(1)
  , \quad\\
  v_f(\xi) = v_*+ o(1)
  , \quad p_f(\xi) = p_*+ o(1)
  , \quad w_f(\xi) = w_*+ o(1)
  , \quad r_f(\xi) = r_*+ o(1),
\end{array}
\end{equation}
and on  $I_4$
\begin{equation}\label{eq.fast_I4_0}
\begin{array}{{l}}
  u_f(\xi) = u_{0}^+(\xi)+ o(1)
  , \quad
  p_f(\xi) = p_{0}^+(\xi)+ o(1)
  , \quad\\
  v_f(\xi) = v_{**} + o(1)
  , \quad p_f(\xi) =  p_{**}+ o(1)
  , \quad w_f(\xi) =  w_{**}+ o(1)
  , \quad r_f(\xi) =  r_{**}+ o(1).
\end{array}
\end{equation}
Observe that the fast expressions $u_f$ of \eqref{eq.fast_I2_0} and \eqref{eq.fast_I4_0} to leading order coincide with \eqref{FAST_new} (with $x^*=-x^{**}$, see further down).

The reduced slow system~(\ref{e:rslow}) on $\mathcal{M}_0^\pm$ 
\eqref{eq.slowman55}
is linear with decoupled $v$
and $w$ dynamics:
\begin{equation}\label{eq:slow0}
\begin{array}{rcl}
v_x &=& q,\\
q_x &=&v \mp 1,\\
w_x &=& \dfrac{r}{D},\\[2mm]
r_x &=& \dfrac{1}{D}(w \mp 1).
\end{array}
\end{equation}
It possesses the hyperbolic equilibria $(v,q,w,r)=\pm(1,0,1,0)$ with
stable and unstable manifolds as sketched in
Figure~\ref{f:reduce_planes}(b) and (c).  The fast system gives
boundary conditions for each of the slow intervals $I_1$, $I_3$, and
$I_5$:
\[
\begin{array}{l}
v(x^*)=v_*, \,\,v(x^{**})=v_{**},\,\, q(x^*)=q_*,
\,\,q(x^{**})=q_{**}, \,\,\\
w(x^*)=w_*, \,\,w(x^{**})=w_{**},\,\, r(x^*)=r_*,
  \,\,r(x^{**})=r_{**},
\end{array}
\]
and the periodicity of the solution gives boundary conditions for $I_1$ and $I_5$
\[
v(-L)=v(L), \quad q(-L)=q(L),\quad w(-L)=w(L), \quad r(-L)=r(L).
\]
Furthermore, because the problem is translation invariant we can set, without loss of generality, 
\[
q(-L)=q(L)=0\,.
\]
Solving the ODEs~\eqref{eq:slow0} with the boundary conditions above
leads to $x^{*}=-x^{**}$ and hence $x^{**}\in(0,L)$ (and $x^{*}\in(-L,0)$). The slow
solutions in lowest order are given in Appendix~\ref{s:slow_approx_params_small}. 
Furthermore, at lowest order, the values of the slow variables in the
fast solution in~\eqref{eq.fast_I2_0} and~\eqref{eq.fast_I4_0} are
\begin{equation}\label{eq.jumpvalues}
\begin{array}{l}
v_*=v_{**}= \dfrac{\sinh(L-2x^{**})}{\sinh(L)},\quad
q_*=-q_{**}= -\dfrac{2\sinh(x^{**})}{\sinh(L)}\,\sinh(L-x^{**}),\\[2.5mm]
w_*=w_{**}= \dfrac{\sinh((L-2x^{**})/D)}{\sinh(L/D),},\quad
r_*=-r_{**}= -\dfrac{2\sinh(x^{**}/D)}{D\sinh(L/D)}\,\sinh((L-x^{**})/D).
\end{array}
\end{equation}
Finally, the jump point $x^{**}$ is determined by the Melnikov
condition for the transition between 
$\mathcal{M}_\varepsilon^-$ and
$\mathcal{M}_\varepsilon^+$ 
-- the persisting locally invariant slow manifolds near \eqref{eq.slowman55}. To find this condition, we use the
Hamiltonian of the full fast system~\eqref{eq.ham}. In the fast
interval $I_4$, the solution jumps from near 
$\mathcal{M}_\varepsilon^-$
to near $\mathcal{M}_\varepsilon^+$. 
Specifically, at the end points
$\xi=\xi^{**}\mp1/\sqrt{\varepsilon}$, with $\xi^{**}=
x^{**}/\varepsilon$ we have
\[
  u=\mp 1 +  o(1),\quad
  p = o(1),\quad
  v= v_{**} +  o(1),\quad
  q = q_{**}+  o(1),\quad
  w= w_{**} +  o(1),\quad
  r = r_{**} +  o(1).
  \]
  The Hamiltonian is constant, hence substitution of the expressions
  above in~\eqref{eq.ham} gives
  \[
    0 =
    H\left(\textstyle\bu\left(\xi^{**}+1/\sqrt{\varepsilon}\right)\right)
    - 
    H\left(\textstyle\bu\left(\xi^{**}-1/\sqrt{\varepsilon}\right)\right)
    =
  2\varepsilon (A_1 v_{**} + B_1 w_{**} +C_1)  + o(\varepsilon).
  \]
Using~\eqref{eq.jumpvalues}, we can conclude that $x^{**}$ has to
satisfy \eqref{e:jump}.

In the parameter space, we analyse the number of spatially periodic solutions of~(\ref{e:2d_system_ODE}) and their stability as given by the Melnikov condition~(\ref{e:jump}) and the stability criterion~(\ref{e:stab_per}), respectively.

\begin{lemma}\label{lem.melnikov_overview}
The number of solutions~$x^{**}$ satisfying the Melnikov
condition~\eqref{e:jump} depends on $C_1$.
\begin{itemize}
\item If $C_1 =0$, then there are always one or three solutions to
  \eqref{e:jump}.  To be specific, $x^{**}=L/2$ always satisfies
  the Melnikov condition. Furthermore, define
  $$B_1^1:=\min\left(-\frac{D\sinh(D/L)}{\sinh(L)}A_1,
    -A_1\right)\,, \quad \textnormal{and}\quad
  B_1^2:= \max\left(-\frac{D\sinh(D/L)}{\sinh(L)}A_1,
    -A_1\right).$$  
    \begin{itemize}
    \item If $B_1\in (B_1^1,B_1^2)$, then there are
  two more solutions $x^{**}$ in $(0,L)$, symmetrically placed around
  $L/2$.  
 \item If $B_1=-A_1$, then there are two
  more solutions at $x^{**}=0$ and $x^{**}=L$.
  \item If
  $B_1 =- \dfrac{D\sinh(D/L)}{\sinh(L)} A_1$, then there is a triple
  solution at $x^{**}=L/2$, i.e., at this $B_1$ value, there
  is a pitchfork bifurcation in the solutions of the Melnikov
  condition.
\item If $B_1 \not \in [B_1^1,B_1^2]$,
  there are no more solutions in $[0,L]$.
  \end{itemize}
  See Figure~\ref{f:exist_ab_plane}(a) for details including co-periodic stability of the solutions as calculated from~(\ref{e:stab_per}). 
  \item If $C_1\neq 0$, then the Melnikov
condition~\eqref{e:jump} is satisfied by either $0$, $1$, $2$ or $3$
    solutions.  Transitions in the number of solutions occur at the
    curves (details are visualised in Figure~\ref{f:exist_ab_plane}(b) including co-periodic stability of the solutions as calculated from~(\ref{e:stab_per})):
    \begin{itemize}
    \item $A_1+B_1 =C_1$, when $x^{**}\to L$ (black curve);
    \item $A_1+B_1 =-C_1$, when $x^{**}\to 0$ (red curve);
    \item  the (blue) parametric curve
\begin{align}
\label{e:aster}
\begin{aligned}
A_1^*(z) =\frac{C_1\,\sinh(L)}{\cosh(z)(D\tanh(z/D)-\tanh(z))}, \\
B_1^*(z) = -\frac{ C_1\,D\sinh(L/D)}{\cosh(z/D)(D\tanh(z/D)-\tanh(z))}, 
\end{aligned}
\end{align}
for $z\in [-L,0)\cup(0,L] $, at which there is a saddle-node
bifurcation and two solutions collide.
    \end{itemize}     
    For $B_1$ and $C_1$ fixed and
    $|A_1|$ large, there is a unique solution
    $x^{**}$ to~\eqref{e:jump}, which satisfies $x^{**} =
 L/2 + C_1 \, \sinh{L}/(2A_1)+ \mathcal{O}(A_1^{-2})$,
  $|A_1|\to\infty$. This implies that $v_{**} = 
  -C_1/A_1+ \mathcal{O}(A_1^{-2})$ for $|A_1|\to\infty$. 
    \end{itemize}   
\end{lemma}

\begin{figure}
	\centering
\includegraphics[width=0.95\linewidth]{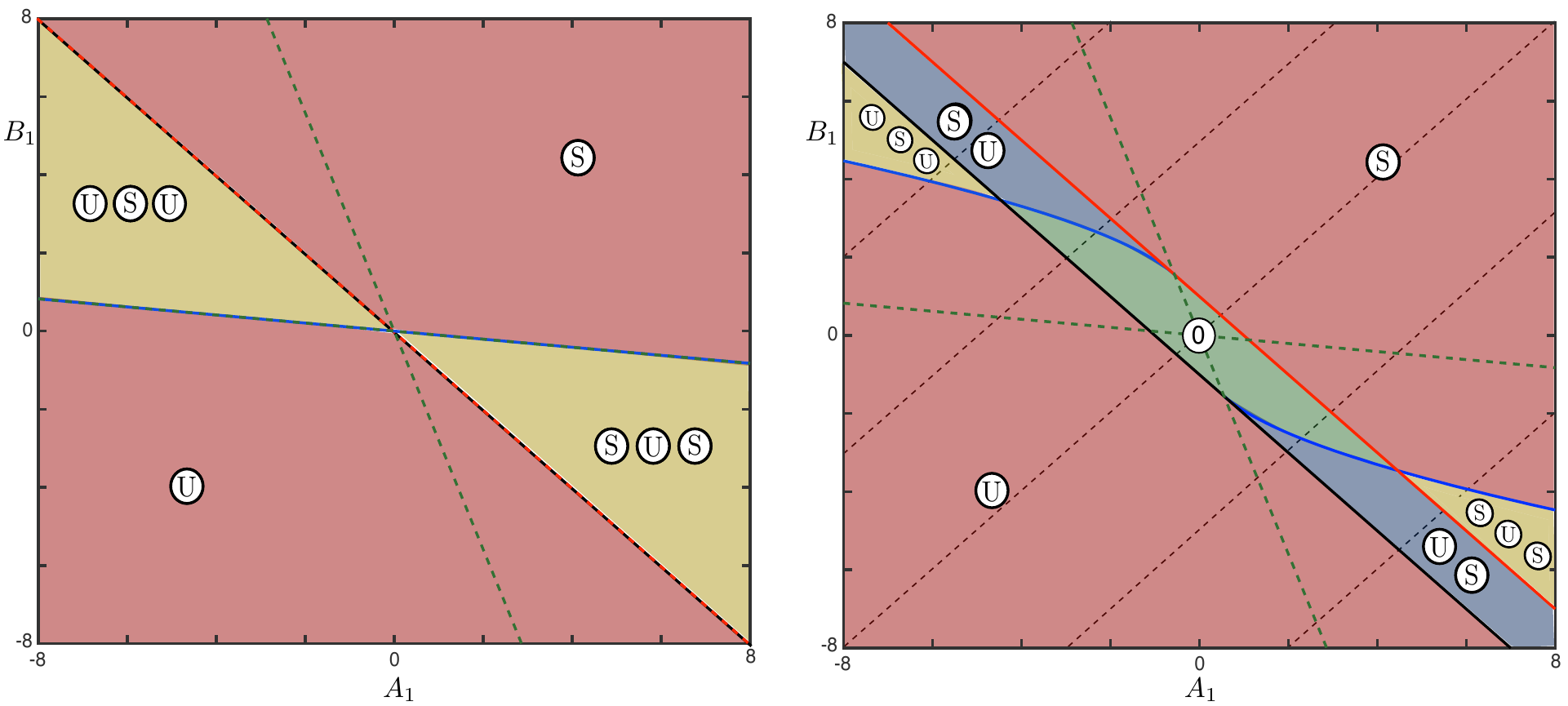}
\caption{
Overview of the stability of the spatially periodic solutions that satisfy the Melnikov condition~\eqref{e:jump} in the  $(A_1,B_1)$-parameter space.  Each {\rm{\raisebox{.5pt}{\textcircled{\raisebox{-.9pt}
      {S}}}}} represents one stable solution and similarly {\rm{\raisebox{.5pt}{\textcircled{\raisebox{-.9pt}
      {U}}}}} represents one unstable solution.
  The left panel shows $C_1=0$ and
  the right panel 
  $C_1=-1$, while the other parameters are fixed at $D=3$ and $L=5$. In the green region there are no roots, in the red regions one, in the blue regions two, and in the yellow regions three. 
  The green dashed curves represent the boundaries of the region in which the
  Melnikov function~$M(z)$
  is non-monotonic on $[-L,L]$, i.e., the green curves are
  $B_1=- A_1 D\tanh(L/D)/ \tanh(L)$ (extremum at $\pm L$)
  and $B_1=-A_1 D\sinh(L/D)/\sinh(L)\,$ (merging of the
  two extrema in the origin).
On the black line ($C_1=A_1+B_1$),
  one of the solutions corresponds to $x^{**}=L$; on the red line
  ($-C_1=A_1+B_1$), one of the solutions corresponds to $x^{**}=0$ (note that the black and the red curve coincide in the left panel). The
  blue curve represents a saddle-node bifurcation corresponding to double solutions $x^{**}\in(0,L]$ (note that the blue and one of the green curves coincide in the left panel), see
  Lemma~\ref{lem.melnikov} for details. 
The termination  
  points of the blue curves on the black/red curves are
  $(A_1,B_1) =(A_1^*(\pm L), B_1^*(\pm L))$, where $A_1^*$
  and $B_1^*$ depend linearly on $C_1$ and nonlinearly on $D$ and $L$, see \eqref{e:aster}. Furthermore, the black/red and blue curves intersect once more at $(A_1,B_1) = 
   (\mp 1-\tilde{b},\pm \tilde{b})$, with $\tilde{b} = 3 (e^{10/3}+1)^2/(e^{10/3}-1)^2 \approx 3.46$.
  The number of roots of~\eqref{e:jump} changes by two when crossing the blue curves and by one when crossing the black and red curves. 
  In Appendix~\ref{S:app_mel}, Figure~\ref{f:bifurcation_a}, bifurcation curves can be found depicting the changes in stability along the dashed black lines in right panel for the the case $C_1=-1$. 
  } \label{f:exist_ab_plane}
\end{figure}

The proof of this lemma involves the analysis of the function $M(z)$
and can be found in Appendix~\ref{S:app_mel}. This Appendix also
contains bifurcation diagrams depicting the changes in stability along
the dashed black lines in right panel for the the case $C_1=-1$
(see Figure~\ref{f:bifurcation_a}). 

Combining the above lemma with the preceding analysis gives the singular limit (i.e., $\varepsilon=0$) existence results as stated in Theorem~\ref{th.two}. What remains to be shown is the persistence of these results for $\varepsilon >0$ small.
 This persistence can be shown by the singular perturbation theory of Fenichel and can be seen as a natural extension of the persistence result for localized $1$-pulse solutions for the three-component reaction-diffusion system~\eqref{e:2d_system} in the same parameter regime, see \S2.2-\S2.4 of \cite{DvHK09} (with $A=\varepsilon A_1=\varepsilon\alpha$, $B=\varepsilon B_1=\varepsilon\beta$ and $C=\varepsilon C_1=\varepsilon\gamma$, see Remark~\ref{XX}). 
 
Before detailing the proof of the persistence of the periodic solutions, first we succinctly describe to persistence proof for the localized $1$-pulse solution. Full details can be found in \cite{DvHK09}.
In \S2.2 of \cite{DvHK09} the authors first derive the singular limit results for localized $1$-pulse solutions (that asymptote to $(-1,-1,-1)+ \mathcal{O}(\varepsilon)$ as $x \to \pm \infty$), that is, they derive the equivalent of the Melnikov condition \eqref{e:jump}, as well as the leading order profiles of the localized solutions. Next, in \S2.4 of \cite{DvHK09} they prove the persistence of such a pattern for $\varepsilon>0$ by showing the existence of a homoclinic orbit $\gamma_{h}(\xi)$ in the fast system (i.e.,~\eqref{e:6fast}) that is contained in the intersection of the stable and unstable manifold of the asymptotic equilibrium point involved and that is in leading order given by the earlier derived profiles. To do so, the authors utilise the reversibility symmetry of the system (i.e., $(\xi,p,q,r) \to -(\xi,p,q,r)$ in \eqref{e:6fast}) and
study both the three-dimensional unstable manifold of the equilibrium point as well as the five dimensional unstable manifold of the (truncated)\footnote{In \cite{DvHK09} there is no need to truncate to locally invariant slow manifolds as the orbits stay away from the fold.} locally invariant slow manifold $\mathcal{M}_\varepsilon^-$ ($u$ near $-1$) as they pass along the other (truncated) slow manifold $\mathcal{M}_\varepsilon^+$ ($u$ near $+1$, see also \eqref{SM}). They show that there is a one-parameter family of heteroclinic orbits in the unstable manifold of the equilibrium point that is forward asymptotic to $\mathcal{M}_\varepsilon^+$. The evolution of such an orbit near $\mathcal{M}_\varepsilon^+$ is governed by the reduced slow system (to leading order given by \eqref{eq:slow0}) and the orbit $\gamma_{h}(\xi)$ of interest is exponentially close to one of these orbits (the one that obeys a version of the Melnikov condition) for an asymptotically long (spatial) time. Next, a three-dimensional tube around this heteroclinic orbit is constructed and this tube is studied as it flows from $\mathcal{M}_\varepsilon^-$ to $\mathcal{M}_\varepsilon^+$ and (partly) back to  $\mathcal{M}_\varepsilon^-$ again. By transversality, which follows from a Melnikov computation, the intersection of this tube with the stable manifold of $\mathcal{M}_\varepsilon^-$ is two-dimensional and, by the reversibility symmetry, orbits in this intersection are close to the part of the stable manifold that are forward asymptotic to the equilibrium point (that is, they are close to the persisting perturbed stable yellow manifolds in panel (b) of Figure~\ref{f:reduce_planes}). What remains to show is that there exists an orbit in this intersection that touches down exactly on this stable manifold.
This follows again from the reversibility symmetry. In particular, a similar two-dimensional object (the intersection of a related three-dimensional tube with the unstable manifold of $\mathcal{M}_\varepsilon^-$) can be constructed and it is subsequently shown that these two two-dimensional objects intersect yielding the existence of the persisting homoclinic orbit $\gamma_{h}(\xi)$.

Next, we extend this proof for localized solutions to periodic solutions. The main difference between a $2L$-periodic solution and a localized solution -- where $L$ is assumed to be sufficiently large to support the slow-fast structure -- is that localized solutions need to asymptote on one of the equilibrium points (that is, they have to lie on the stable and unstable manifolds of the equilibrium point), while this is not the case for a $2L$-periodic solution. Instead the periodic solutions are fixed by the requirement that they  
have zero derivative at the matching point $\pm L$. For the slow components (and in the singular limit $\varepsilon=0$) this difference is indicated in the phase planes of panels (b) and (c) of Figure~\ref{f:reduce_planes}: localized solutions need to lie (asymptotically) close to the yellow stable and unstable manifolds, while $2L$-periodic solutions are indicated by the blue orbits intersecting $\{q,r=0\}$. 
In particular, 
to prove the persistence of a $2L$-periodic pattern for $\varepsilon>0$ one needs to show the existence of a periodic orbit $\gamma_{P}(\xi)$ in the fast system~\eqref{e:6fast} that is contained in the forward and backward flow of the three-dimensional hyperplane 
$\{p=0, q=0, r=0\}$ in the neighbourhood of the persisting equilibrium with $\ueq = -1 + \mathcal{O}(\varepsilon)$. 
However, by the scale separation and the linear nature of the reduced slow system (to leading order given by \eqref{eq:slow0}), the forward and backward flows of this three-dimensional hyperplane will be asymptotically close to the related three-dimensional manifolds of interest for the localized pattern
while they make the transition to the other slow manifold.
That is, the information of the flow of the five dimensional unstable manifold of the truncated locally invariant slow manifold~$\mathcal{M}_{\varepsilon,\delta}^-$ as it passes along the other truncated slow manifold $\mathcal{M}_{\varepsilon,\delta}^+$ (with $u$ near $+1$, see also~\eqref{SM}) can still be utilised, similarly for the stable manifold.
Furthermore, the reversibility symmetry of \eqref{e:6fast} still holds. 
As a result, the proof of the persistence of the localized solution from \cite{DvHK09} (and as outlined above) only needs to be adjusted slightly.
We omit further technical details and refer to~\cite{doelman1997pattern}, where a similar adjusted proof is given for periodic patterns in the one-dimensional Gray-Scott model. This completes the proof of the existence part of Theorem~\ref{th.two}. 

For the proof of the stability result, we refer to~\cite{H18}, where the authors derive a stability criterion for the periodic orbits under perturbations with the same period (known as co-periodic stability). In our notation the condition reads as~\eqref{e:stab_per}.
$\hfill{\Box}$

After finishing the proof of Theorem~\ref{th.two}, we reflect on the case $C=0$ and the transition to $C=\varepsilon C_1$. 
From Theorem~\ref{th.two} and Lemma~\ref{lem.melnikov_overview} it follows that, for $\varepsilon$ small enough and if $A=\varepsilon A_1$, $B=\varepsilon B_1$ and $C=0$, then there exists a symmetric
  periodic solution with fast transitions at $x=\pm L/2$ and the slow components $v,w$ are to leading order zero during
  the fast transitions, see panel \raisebox{.5pt}{\textcircled{\raisebox{-.9pt}
      {5}}} in
  Figure~\ref{fig_overview}. 
  Furthermore, there are two more periodic
  solutions when $B_1 \in
  (-A_1 D\sinh(L/D)/\sinh(L) , -A_1)$, see panel \raisebox{.5pt}{\textcircled{\raisebox{-.9pt}
      {4}}} in Figure~\ref{fig_overview} for a typical example. At
  $B_1= -\ A_1D\sinh(L/D)/\sinh(L)+o(1)$, these solutions get
  created in a symmetric pitchfork bifurcation at the solution with the fast
  transitions at $x=\pm L/2$. At $B_1=-A_1+o(1)$, these solutions cease
  to exist as the transition points start approaching $x=0$ or
  $x=\pm L$. 

The symmetric pitchfork bifurcation breaks open and there are two curves of periodic solutions with two fast transitions when $C=\varepsilon C_1$ with $C_1\neq 0$, see Figure~\ref{fig_overview22}.
This results in regions in $(A_1, B_1)$-parameter space with $0, 1, 2$ or $3$ periodic solutions with two fast transitions. 
The transition between the different regions are determined by \eqref{e:aster} and the curves $A_1+B_1=\pm C_1$.
 Furthermore, we observe that there are no solutions when $A_1$ and $B_1$ are too small compared to $C_1$, see Figure~\ref{f:exist_ab_plane}. For instance, a necessary -- but not sufficient -- condition for the existence of periodic solutions with two fast transitions in this parameter regime is $|A_1| + |B_1| > |C_1|$. 
From Figure~\ref{f:exist_ab_plane} it also follows that the number of
supported periodic solutions with two fast transitions depends
intrinsically on both $A_1$ and $B_1$. For instance, it is not possible to have three different
periodic solutions with two fast transitions for $B=0=B_1$. For
$B=0$, system \eqref{e:2d_system} effectively reduces to a
two-component model. Hence, 
the existence of three different periodic solutions with two fast transitions
requires the three-component system \eqref{e:2d_system} and is not present in the simpler two-component model.
 
 Note that when $x^{**}\to L$, the leading order value of $v$ during the fast transitions approaches minus one (i.e., $v_{**}\to -1$) and the two heteroclinic
  connections are very close together, hence one of the slow intervals
  becomes very small. This type of solution is in lowest order similar
  to the limiting solutions with one fast jump seen in the previous
  section when $A\to 0$. The same holds for $x^{**}\to 0$ as then
  $v_{**}\to 1$.
Furthermore, for $L \to \infty$ the Melnikov condition \eqref{e:jump} approaches the existence condition for stationary localized $1$-pulse solutions (i.e., a periodic solution with the two fast transitions and infinite period) as constructed in \cite{DvHK09}\footnote{In \cite{DvHK09} the stationary localized $1$-pulse solutions asymptote to $-1$, while the constructed period solutions in this paper approach $+1$ at the boundary $\pm L$, see Figure~\ref{f:periodic_setup}. Hence, by the symmetry $(U,V,W,C) \mapsto (-U,-V,-W,-C)$ of the system we actually have that the Melnikov condition \eqref{e:jump} for $L \to \infty$ approaches existence condition of \cite{DvHK09} with $\gamma$ replaced by $-\gamma$.}.
Finally, from \eqref{eq.jumpvalues} and Lemma~\ref{lem.melnikov_overview} it follows that for $A_1 \to
  \infty$ (with $B_1, C_1$ fixed) the transition points of the unique periodic solution
 approach  $\pm L/2$ (i.e., $x^{**} \to L/2$), and hence $v_{**} \to 0$. That is, they connect to the periodic solutions with two fast jumps at $A_0\neq0$, see Figures~\ref{fig_overview} and \ref{fig_overview22} and the next section for more details.

\subsection{Proof of Theorem~\ref{th.two_2}: $A_0\neq 0$ and a fast jump at $v_0=0$}
\label{S:SFS_Alarge}
Next we look at slow-fast periodic solutions with two fast transitions where
$A_0\neq 0$, while keeping $B_0=0=C_0$. That is, Assumption~\ref{H1} holds.
 The $u$-component of the
hyperbolic parts of the slow manifold are characterized by $v$:
\[
\mathcal{M}_0^\pm= \left\{(\hu_0^\pm(v),0,v,q,w,r) \:\mid\: 
  (\hu^\pm_0)^3-\hu^\pm_0+A_0v=0; \,\pm\hu_0^\pm>\frac1{\sqrt3}; \,\, 
v,q,w,r \in\mathbb{R}^4\right\}.
\]
Contrary to the previous section, the function
$K(v,w)=A_0v\equiv \hu^\pm_0-(\hu^\pm_0)^3$ can -- and will -- vary during the
slow evolution. A heteroclinic fast transition between the two slow manifolds can only occur when $K=0$, hence
when $v=0$, see Figure~\ref{fig.fast_dynamics} and Lemma~\ref{l:fd}. During the fast phases $I_2$ and $I_4$, the slow variables
are again constant in lowest order and we denote the lowest order approximation
of the slow variables by $v_*=0$, $q_*,w_*,r_*$ in $I_2$ and
$v_{**}=0$, $q_{**},w_{**},r_{**}$ respectively in $I_4$. Moreover, without loss of generality, we assume that $p<0$ in $I_2$ and $p>0$ in $I_4$, that is, in $I_2$ the fast component $u$ jumps down, while it jumps up in $I_4$, and
$q(-L)=q(L)=0$.  

During the slow phase $I_1\cup I_3\cup I_5$, the slow dynamics in
lowest order is given by
\begin{equation}\label{eq.slow_5.2}
(v_s)_x = q_s, \quad (q_s)_x = v_s-\hu^\pm_0(v_s) ,\quad
  D (w_s)_x = r_s, \quad D(r_s)_x = w_s-\hu^\pm_0(v_s) ,  
\end{equation}
where $\hu_0^\pm(v_s)$ lies on the slow manifold
$\mathcal{M}_0^\pm$. At the end points of the slow phase intervals
(i.e., at $x^*$ and $x^{**}$), the $v_s$ orbits must approach $v=0$,
hence $\hu_0^\pm(v_s)$ approaches $\pm 1$, see, for instance, panels
\raisebox{.5pt}{\textcircled{\raisebox{-.9pt} {6}}} and
\raisebox{.5pt}{\textcircled{\raisebox{-.9pt} {7}}} of
Figure~\ref{fig_overview}.  In the approximate slow system~\eqref{eq.slow_5.2}, the
$(v_s,q_s)$-component decouples from the $(w_s,r_s)$ one. Once the
$(v_s,q_s)$-component is solved, the remaining $(w_s,r_s)$-component
is a linear non-autonomous system, hence it can be solved explicitly
in terms of~$v_s$.

We focus first on the $(v_s,q_s)$-system and drop the index
``s'' for the time being.  As we have seen before (see also Figure~\ref{fig_uK}), if
$K=A_0v<2/(3\sqrt3)$, then there is a unique
$u_0^+(v)\in\mathcal{M}_0^+$ with $u_0^+(v)>1/\sqrt3$ and if
$K=A_0v>-2(3\sqrt3)$, then there is a unique
$u_0^-(v)\in\mathcal{M}_0^-$ with $u_0^-(v)<-1/\sqrt3$. Hence,
the relation $u=u_0^\pm(v)$ is a bijection between
$\left\{v \mid \pm A_0v \leq 2/(3\sqrt3) \right\}$ and
$\left\{ u \mid \pm u \geq 1/\sqrt3\right\}$.
By using this bijection to change from the slow $v$ variable to the $u$ variable, we get an explicit slow system on $I_1\cup I_3\cup I_5$
given by \eqref{eq.uq} and with the boundary conditions $|u|\to 1$ for $x \to x^*$ and $x\to x^{**}$.
The $(u,q)$-system \eqref{eq.uq} is Hamiltonian and this leads to a conserved
quantity given by
\[
  E(u,q) = \frac{A_0^2q^2}{2} + V(u), \quad\mbox{with}\quad
  V(u):=\frac{u^2}{4}\left( A_0(2-3u^2)-2(u^2-1)^2\right).
\]
Since $V'(u) = u(1-3u^2)(u^2-1+A_0)$, the conserved quantity $E$ has
extrema at the fixed points $(u,q)=(0,0)$ and
$(u,q)=\left(\pm\sqrt{1-A_0},0\right)$ and the degenerate points
$(u,q)=\left(\pm1/\sqrt3,0\right)$.
From the linearisation about the fixed and degenerate points, it
follows that for
\begin{itemize}
\item $A_0>2/3$: $E(u,q)$ has saddles at $(u,q)=\left(\pm\dfrac1{\sqrt3},0\right)$;\footnote{The fixed points with $|u|<1/\sqrt{3}$ are not of interest for the slow dynamics. \label{foot}}
\item $A_0<\dfrac23$: $E(u,q)$ has saddles at
  $(u,q)=\left(\pm\sqrt{1-A_0},0\right)$ and minima at
$(u,q)=\left(\pm\dfrac1{\sqrt3},0\right)$.$^{\ref{foot}}$
\end{itemize}
\begin{figure}
	\centering
                \includegraphics[width=0.9\linewidth]{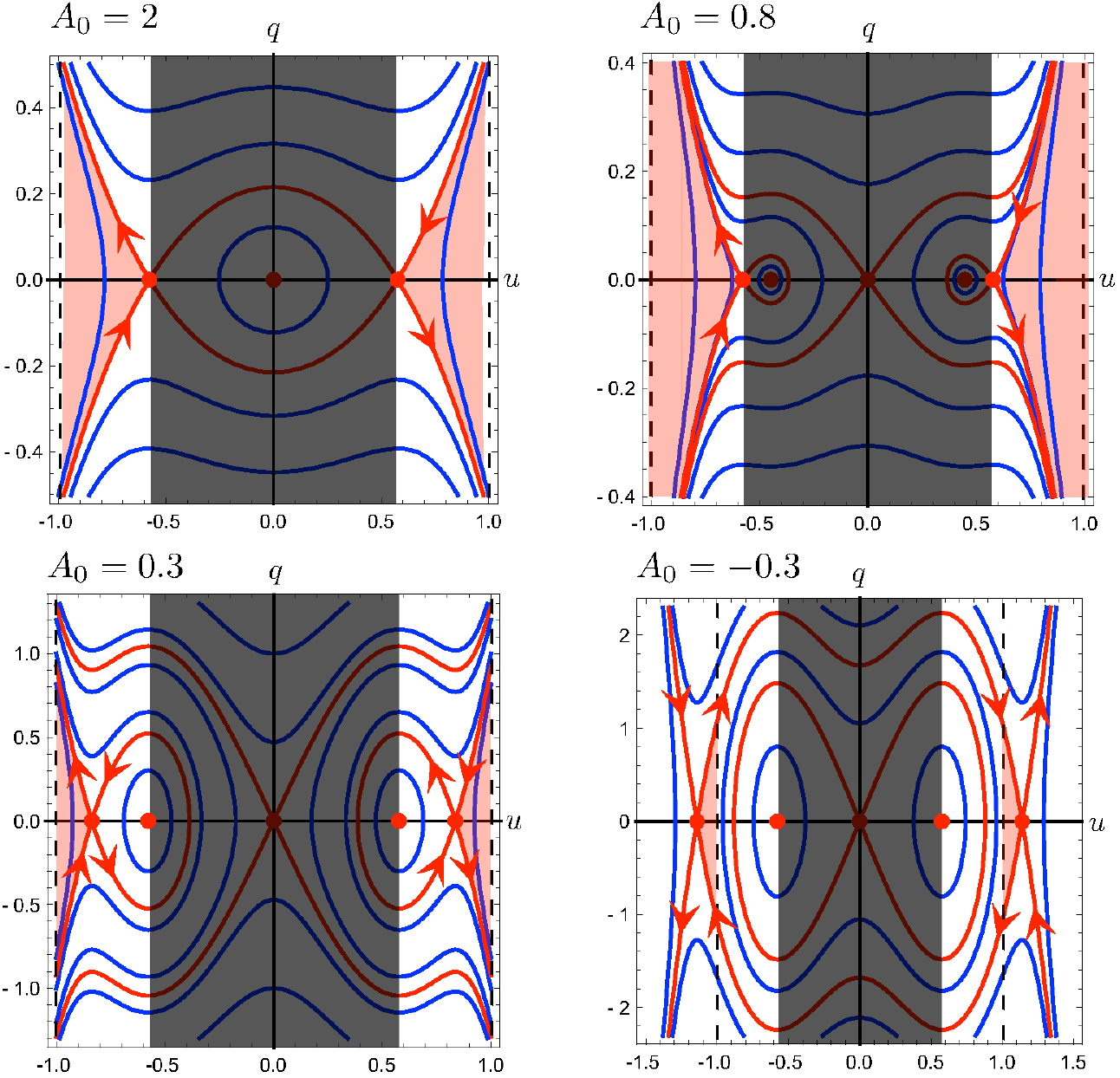}
	\caption{Phase plane of the singular slow
          system \eqref{eq.uq} for four different $A_0$ values. Note that \eqref{eq.uq} is not defined in the grey areas as $|u| <  \dfrac1{\sqrt3}$. The black dashed lines indicate $u = \pm 1$ and the red areas will be of interest in the construction of the periodic solution.    
          \label{f:sing_slow}
          }
\end{figure}
The level sets of $E$ are illustrated in Figure~\ref{f:sing_slow} and
only the regions with $|u|>1/\sqrt3$ have relevance for the slow
dynamics.  The irrelevant regions where $|u|<1/\sqrt3$ are therefore shaded grey in   Figure~\ref{f:sing_slow}.
Every orbit of the slow dynamics~\eqref{eq.uq} lies on a
level set of $E$ and in Figure~\ref{f:sing_slow} also the direction of
the flow is indicated.  The periodic solutions must satisfy the
boundary condition $|u|\to 1$ for $x\to x^*$ and $x\to x^{**}$. Thus
both slow orbits lie on the level set with value
\begin{eqnarray}
\label{ESTAR}
  E^* = \frac{A_0}4 \left(2A_0q_*^2-1\right).
\end{eqnarray}
Only the orbits which move from $u=\pm 1$ towards a saddle and back
towards $u=\pm1$ (hence with the same sign of $u$) are relevant. The
regions these orbits lie in are shaded red in Figure~\ref{f:sing_slow}
and the black dashed lines indicate $u=\pm 1$. That is, the slow parts
of the periodic orbits under construction have to lie in the red
shaded regions in Figure~\ref{f:sing_slow}.  The symmetry of the $E$
level sets imply that $q_*=-q_{**}<0$. Thus in $I_3$, the orbit goes
from $(-1,q_*)$ at $x^*$ to $(-1,-q_*)$ at $x^{**}$, in $I_1$, the
orbit goes from $(u(-L),0)$ at $-L$ to $(1,q_*)$ at $x^*$, and in $I_5$ the
orbit goes from $(1,-q_*)$ at $x^{**}$ to $(u(-L),0)$ at $L$,
see Figures~\ref{f:periodic_setup} and ~\ref{f:sing_slow}. 
Due to the rotational symmetry of order two of system~\eqref{eq.uq} the {\emph{time of flight}} of both
orbits is the same, that is, the spatial time $ x^{**}-x^*$ spend on $I_3$ is the same as 
the spatial time $ 2L-(x^{**}-x^*)$ spend on $I_1\cup I_5$.
Combined with the boundary conditions this implies that $x^*=-L/2$.
In addition, the
    orbits in $I_3$ and $I_1\cup I_5$ are related by symmetry:
    for $x\in I_1$ we have $(u,q)(x) = -(u,q)(x+L) \in I_3$ and for $x\in I_5$ it holds that $(u,q)(x) = -(u,q)(x-L) \in I_3$.

Using the relation between $u_x$ and $q$
  in~\eqref{eq.uq} and substituting this into the definition of $E$,
  we get on $I_3$ an initial value problem for $u$
    \[
      E^* = \frac{(u_x)^2(1-3u^2)^2}{2} +V(u), \mbox{ hence }
      u_x = \pm \frac{\sqrt{2(E^*-V(u))}}{3u^2-1}, \quad u(x^*)=-1.
    \]
    This can be rewritten to get an implicit relation for $u(x)$ and
    the relation between $q_*$ and $L$. The further details depend on
    the value of $A_0$ and we need to distinguish three cases.
\begin{itemize}
\item $A_0<0$. In this case the saddles are at $(\pm\sqrt{1-A_0},0)$, hence
  the $u$-components have an absolute value greater than~$1$, see the bottom right plot of Figure~\ref{f:sing_slow}. The orbits connecting $(-1,q_*)$ -- recall that $q_*<0$ -- with $(-1,-q_*)$
  must have an $E$ value less than $E(-\sqrt{1-A_0},0)= - A_0 (A_0-1)^2/4$, hence
  $-q_*\in \left(0, \sqrt{(2-A_0)/2}\right)$. 
  The minimal $u$ value is
    attained at $x=0$ when $q=0$ and is implicitly given by
  \begin{align*}
      &E(u^*_{\rm min},0) = E^*, \quad -\sqrt{1-A_0}< u_{\rm min}^*<-1 \implies \\
      &\frac{(u^*_{\rm min})^2}{4}\left( A_0(2-3(u^*_{\rm min})^2)-2((u^*_{\rm min})^2-1)^2\right) = \frac{A_0}4
      \left(2A_0q_*^2-1\right), \quad \\
      &  \qquad \qquad \qquad \qquad \qquad \qquad \qquad \qquad \qquad \qquad \qquad-\sqrt{1-A_0}<u_{\rm min}^*<-1.
\end{align*}
    The ODE for $u$ implies the implicit equation for $u(x)$ for $x>0$
    and $x\in I_3$
    \[
      x = \int_{u^*_{\rm min}}^{u(x)}\frac{3u^2-1}{\sqrt{2(E^*-V(u))}}\, du \,,
    \]
    and by taking $x=L/2$ and recalling \eqref{ESTAR}, it gives the relation between $q_*$
    and $L$
    \[
      L = 2\int_{u^*_{\rm min}}^{-1} \frac{3u^2-1}{\sqrt{2(E^*-V(u))}}\, du\,.
    \]
    If $q_*$ goes from 0 to $-\sqrt{(2-A_0)/2}$ then $L$
  goes from 0 to $\infty$. \\
\item $0<A_0<2/3$. Again, the saddles are at $(\pm\sqrt{1-A_0},0)$,
  but now the $u$-components have an absolute value less than~$1$, see the bottom left plot of Figure~\ref{f:sing_slow}.
  The bound on the values of $E$ still gives that $-q_*\in \left(0,
    \sqrt{(2-A_0)/2}\right)$, but now we get a maximal value of
  $u$ on the orbit, given by the relation 
      \begin{align*}     
      &\frac{(u^*_{\rm max})^2}{4}\left( A_0(2-3(u^*_{\rm max})^2)-2((u^*_{\rm max})^2-1)^2\right) = \frac{A_0}4
      \left(2A_0q_*^2-1\right), \quad \\
      &  \qquad \qquad \qquad \qquad \qquad \qquad \qquad \qquad \qquad \qquad \qquad-1<u_{\rm max}^*<-\sqrt{1-A_0}.
\end{align*}
    The  implicit equation for $u(x)$ for $x>0$
    and $x\in I_3$ is
    \[
      x = \int^{u^*_{\rm max}}_{u(x)}\frac{3u^2-1}{\sqrt{2(E^*-V(u))}}\, du  \,,
    \]
    and  the relation between $q_*$
    and $L$ is
    \[
      L = 2\int^{u^*_{\rm max}}_{-1} \frac{3u^2-1}{\sqrt{2(E^*-V(u))}}\, du \,.
    \]
    As before, if $-q_*$ goes from 0 to $\sqrt{(2-A_0)/2}$ then $L$
    goes from 0 to $\infty$.\\
  \item $A_0\geq 2/3$. Now the saddle is at the degenerate point
    $\left(1/\sqrt3,0\right)$, see the top plots of Figure~\ref{f:sing_slow}.   
  Since $E(1/\sqrt3,0)=-2/27+A_0/12$, the bound on the values of $E$ gives that $-q_*\in \left(0,
    \sqrt{\dfrac{2(9A_0-2)}{27A_0^2}}\right)$.  The maximum value of
  $u$ on the orbit is given by the relation
          \begin{align*}     
      &\frac{(u^*_{\rm max})^2}{4}\left( A_0(2-3(u^*_{\rm max})^2)-2((u^*_{\rm max})^2-1)^2\right) = \frac{A_0}4
      \left(2A_0q_*^2-1\right), \quad \\
      &  \qquad \qquad \qquad \qquad \qquad \qquad \qquad \qquad \qquad \qquad \qquad u_{\rm max}^*<-1/\sqrt3.
\end{align*}
    The  implicit equation for $u(x)$ for $x>0$
    and $x\in I_3$ is
    \[
      x = \int^{u^*_{\rm max}}_{u(x)}\frac{3u^2-1}{\sqrt{2(E^*-V(u))}}\, du \,,
    \]
    and  the relation between $q_*$
    and $L$ is
    \[
      L = 2\int^{u^*_{\rm max}}_{-1} \frac{3u^2-1}{\sqrt{2(E^*-V(u))}}\, du \,.
    \]
    However, if $-q_*$ goes to $\sqrt{\dfrac{2(9A_0-2)}{27A_0^2}}$, then
    $u^*_{\rm max}$ goes to the singular value $1/\sqrt3$ and
    $E^*-V(u)$ goes to $(3u^2-1)^2(6u^2+9A_0-8)/108$. Thus the
    $L$ integral loses its singularity at $u=1/\sqrt 3$ and the
    length function is bounded with the maximal length given by \eqref{LMAX_new}   
     \begin{align*} 
   L_{\rm max}(A_0) =   6\sqrt6\int^{-\frac1{\sqrt3}}_{-1} \frac{du}{\sqrt{6u^2+9A_0-8}}  = 
      6\log\left(\frac{\sqrt6+\sqrt{9A_0-2}}{\sqrt2+\sqrt{9A_0-6}}\right).
    \end{align*}
    This expression is monotonically decreasing in $A_0$, decays to 0
    for $A_0\to \infty$ and 
    $L_{\rm max}(2/3)$ $= 6\ln\left(\sqrt2+\sqrt3\right)\approx 6.9$. See
    Figure~\ref{fig.Lmax} for a sketch of this function. Note that
    $L_{\rm max}$ has a vertical derivative at $A_0=2/3$.
\begin{figure}
	\centering
\includegraphics[width=0.5\linewidth]{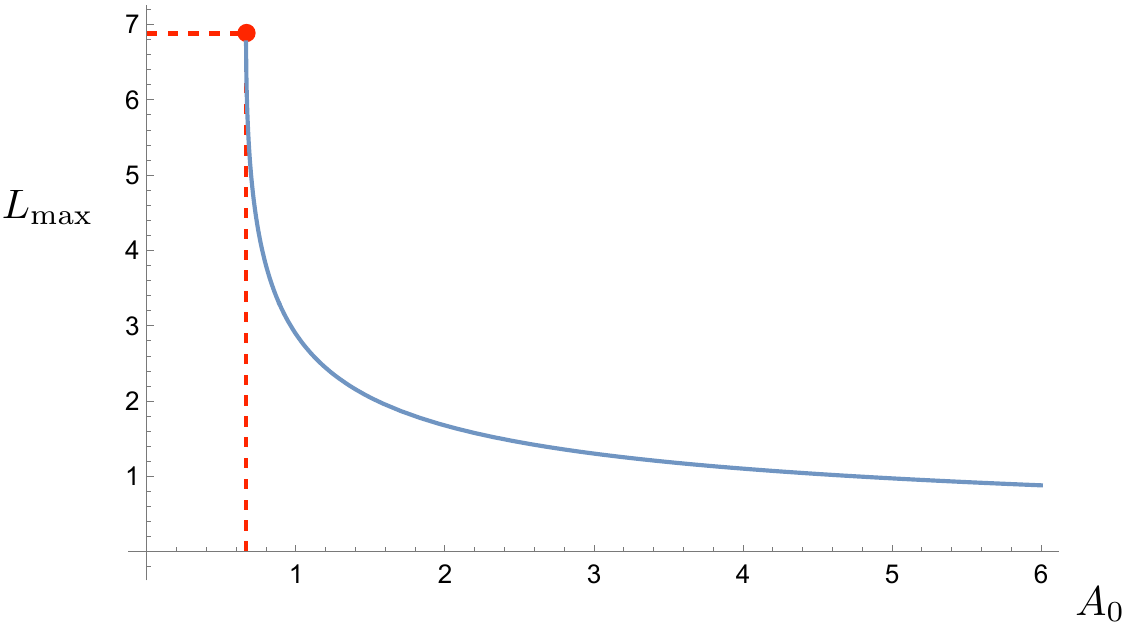}
        	\caption{Plot of $L_{\rm max}$ as function of $A_0$
         \label{fig.Lmax}}
\end{figure}
This explains the turning point in the continuation depicted in the bifurcation diagram in Figure~\ref{fig_overview} at point \raisebox{.5pt}{\textcircled{\raisebox{-.9pt} {8}}} and
at point \raisebox{.5pt}{\textcircled{\raisebox{-.9pt} {6}}} 
in the bifurcation diagram in Figure~\ref{fig_overview22}. From the expression above, it also
    follows that the maximal $A_0$ value for $L=5$ is $A_0\approx
    0.70$.  For $A_0$ close to this maximal value,  $u_{\rm max}^*$ gets
    close to $-1/\sqrt3$ and the $(u,q)$-system~\eqref{eq.uq}
    becomes fast due to the degeneracy and hence a new type of
    solution will start, as can also be seen in panel \raisebox{.5pt}{\textcircled{\raisebox{-.9pt} {8}}} of Figure~\ref{fig_overview} and panel \raisebox{.5pt}{\textcircled{\raisebox{-.9pt} {6}}} of Figure~\ref{fig_overview22}. This is why the critical value is not fully
    reached. \\
\end{itemize}

Now we have established the slow dynamics in the $(u,q)$-system, and hence the $(v,q)$-system, we
can solve the $(w,r)$-system. Using the method of variation of
parameters to solve the inhomogeneous linear ODE and the continuity 
conditions at $x=\pm L/2$ and $x=\pm L$, we get \eqref{slowTH2} for $x\in I_3 $ and \eqref{slowTH22} for the other two slow intervals.
Note that this implies that during the fast phase $w_*=w_{**}=0$ and $r_*=-r_{**}<0$.

Thus far, we have described the lowest order heuristics for the periodic solutions with two fast transitions with $A_0 \neq 0$. Again, a Melnikov function and the singular perturbation theory of Fenichel can be used to prove the persistence of these periodic solutions for $0<\varepsilon \ll 1$. As this is similar in spirit to the proofs for the other two types of periodic solutions constructed before we omit these details.

\section{Discussion}\label{S:discussion} 
In this paper, we studied stationary periodic solutions in a one-dimensional singularly perturbed three-component reaction-diffusion system~\eqref{e:2d_system}. The model was originally developed as a  phenomenological model of gas-discharge dynamics~\cite{P98,P14,P97}. Subsequently, various rigorous existence and stability results of localized states were proven in a series of papers~\cite{CBDvHR15,CBvHIR19,DvHK09,NTU03,R13,TvH21,H18,vHDK08,vHDKP10,vHS11,vHS14}. These results were derived, however, in the parameter regime where the coupling between the slow $v,w$ components with the fast $u$ component is not too strong, i.e., $A, B$ and $C$ in \eqref{e:2d_system} were of order $\varepsilon$. In this paper, we expanded the parameter regime and allowed the parameter $A$ to range from small to order $1$, while keeping the parameters $B$ and $C$ small\footnote{The role of the parameters $A$ and $B$ are interchangeable and similar results can thus be obtained for varying $B$ while keeping $A$ and $C$ small.}. Moreover, in contrast to most previous studies, we analysed periodic solutions instead of localized states.

We showed how near-equilibrium periodic patterns emerge through a
Turing instability and evolve to various far-from equilibrium $2L$-periodic patterns by varying $A$ from order $1$ to small. That is, we showed how the near-equilibrium periodic patterns and far-from equilibrium periodic patterns are connected. In particular, we used techniques from singular perturbation theory to show how a periodic solution with one fast transition emerges through a
Hamiltonian-Hopf bifurcation from the trivial solution for $A$ near $2/3$, see Lemma~\ref{L:HH}. This periodic solution starts as near-equilibrium periodic pattern with a small amplitude, but grows, for decreasing $A$, to a far-from equilibrium periodic pattern with one homoclinic fast transition; see Theorem~\ref{th.one} for the details. Upon decreasing $A$ further to order $\varepsilon$, the periodic solution transforms into a periodic solution with two heteroclinic
 fast transitions. The width of this periodic solution is to leading order determined by the solutions of the Melnikov condition \eqref{e:jump} and is described in Theorem~\ref{th.two}. Note that these periodic patterns are closely related to localized states studied previously in, for instance, \cite{DvHK09}. Upon further decreasing, or increasing, $A$ back to order $1$ the periodic solution with two fast transitions transforms into a different type of periodic solution with two fast transitions; see Theorem~\ref{th.two_2}.

\subsection{Coexistence of multiple periodic solutions}
For a fixed $L$ there is a maximum value $A_{\rm max}$ implicitly defined by $L=L_{\rm max} (A_{\rm max})$ \eqref{LMAX_new} such that a $2L$-periodic solution with two 
 fast transitions and width $L$ ceases to exist upon increasing $A$ to $A_{\rm max}$.  
As a result, the solution branch obtained by the numerical continuation program AUTO-07P \cite{auto} turns around when ones try to continue $A$ past $A_{\rm max}(L)$. We observe that (an) additional small fast transition(s), related to a small homoclinic orbit in the fast system, appears in the solution; see panels
\raisebox{.5pt}{\textcircled{\raisebox{-.9pt}
    {8}}} of
Figure~\ref{fig_overview} and panels \raisebox{.5pt}{\textcircled{\raisebox{-.9pt}
    {6}}}-\raisebox{.5pt}{\textcircled{\raisebox{-.9pt}
    {8}}} of
Figure~\ref{fig_overview22} and Figures~\ref{fig.fast_dynamics} and \ref{split}. 
\begin{figure}
	\centering
	\includegraphics[width=\linewidth]{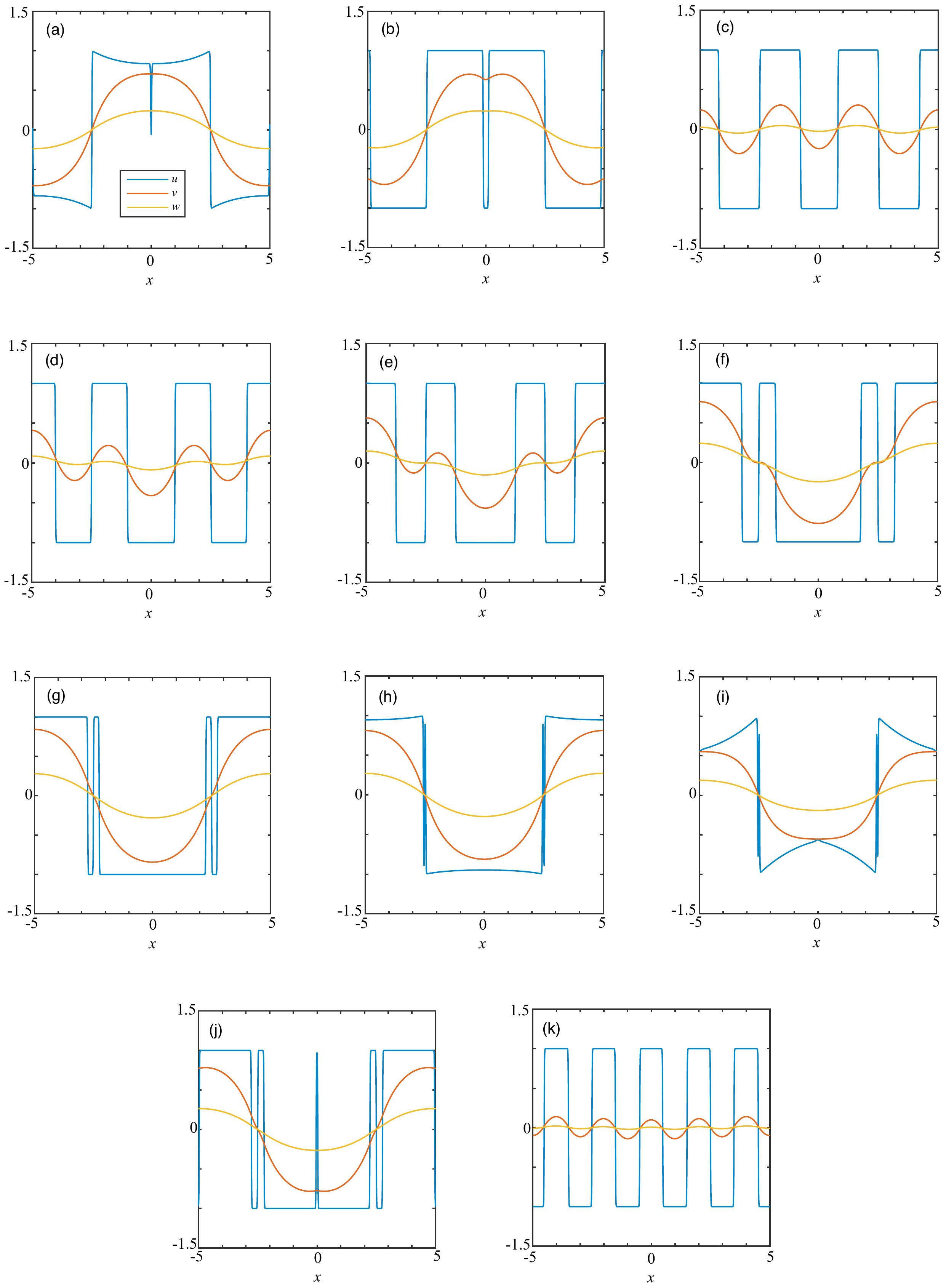}
	\caption{Figures showing the splitting of the periodic patterns for $B=0.01,C=0,\varepsilon=0.01,D=3$ and (a) $A=0.0351698$ (b) $A=-0.00362148$ (c) $A=-0.00478936$ (d) $A=-0.00487184$ (e) $A=-0.00474795$ (f) $A=-0.00388322$ (g) $A=-0.00285898$ (h) $A=0.116642$ (i) $A=0.692597$ (j) $A=-0.00287210$ (k) $A=-0.00372264$}
	\label{split}
\end{figure}
When $C\neq0$ (Figure~\ref{fig_overview22}), there is one new fast transition in the center of the domain and when $C=0$ (Figure~\ref{split}), there are two new fast transitions: one in the center of the domain and one on the boundaries. The two fast transitions are related by the additional symmetry $(U,V,W) \to -(U,V,W)$ of the system with $C=0$.
The new periodic solution with one extra fast transition can be seen as concatenations of a $2L$-periodic solution with two 
 fast transitions and width $L$ from Theorem~\ref{th.two_2} and an $L$-periodic solution with one 
 fast transition from Theorem~\ref{th.one}. As such, it is expected that these new type of solutions can also be analysed with the techniques from this paper. The other new type of solutions can be described and analysed  in a similar fashion.
 However, we decided to not pursue this direction in the current paper.

If one keeps on decreasing $A$ back to order $\varepsilon$ each new homoclinic fast transition again transforms into two heteroclinic fast transitions and we thus observe the formation of a $2L$-periodic solution with four or six fast transitions, see Figure~\ref{split}. 
These new periodic solutions can, in principle, again be explicitly studied using the earlier techniques of this paper. Upon further continuing $A$ this process of adding fast transitions and adjusting interface locations continues and this process is reminiscent of homoclinic snaking \cite{avitabile2010snake, beck2009,Burke}. It would be interesting to further research this potential connection.

\subsection{Future Research}

We have shown that for a given set of parameters system
\eqref{e:2d_system} supports a multiple of stationary periodic
solutions with different characteristics. When all parameters are
small, we have also determined their co-periodic stability/instability
using the action functional approach from~\cite{H18}. A natural next step
would be to study the stability of the other periodic patterns to see
which of these are {\emph{observable}}, e.g., by trying to extend the
action functional method or by using  Evans function techniques
from for instance \cite{Harmen}. These stability results, combined with the results from this paper, would form the starting point for analysing and understanding the dynamic properties of non-stationary periodic patterns. That is, how do initial conditions (with certain properties) evolve towards the stable stationary periodic solutions? For localized states and with all parameters small this was done in \cite{vHDKP10}.

This paper can also be seen as the foundation for further work on
the analysis of planar grain boundaries -- where two differently
orientated spatially periodic patterns meet on the plane -- that
requires a sound knowledge of the existence and (transverse) stability
properties of periodic one-dimensional patterns; see for instance
\cite{Haragus2007,Haragus2012,Scheel2014} in the context of the Turing pattern forming systems. 

\section*{Acknowledgments}
The authors thank K. Harley and A. Doelman for fruitful discussions.
PvH and GD acknowledge that a crucial part of this paper was established during the first joint Australia-Japan workshop on dynamical systems with applications in life sciences. GD thanks Queensland University of Technology for their hospitality.

\appendix

\section{Slow approximation for all parameters small}\label{s:slow_approx_params_small}
The approximation of the slow solutions for all parameters small in Theorem~\ref{th.two} and section~\ref{S:SFS_Asmall}.

\begin{itemize}
\item For $x\in I_1$ \eqref{intervals}: $u_s(x)=1+ o(1)$, 
  $p_s(x)=0+ o(1)$,
\begin{equation}\label{e:I1}
\begin{array}{rcl rcl}
v_s (x)&=& 1 -\dfrac{2\sinh(x^{**})}{\sinh(L)}\,\cosh(L+x)+
             o(1),
\\[2mm] q_s (x)&=& -\dfrac{2\sinh(x^{**})}{\sinh(L)}\,\sinh(L+x)+
              o(1),
  \\[2mm] 
w_s (x)&=& 1 - \dfrac{2\sinh\left(x^{**}/D\right)}{\sinh\left(
             L/D\right)}\,
             \cosh\left((L+x)/D\right) + o(1),
\\[2mm] r_s (x)&=& -\dfrac{2\sinh\left(x^{**}/D\right)}{D\sinh\left(
             L/D\right)}\,\sinh\left((L+x)/D)\right)+ o(1).
\end{array}
\end{equation}
\item For $x\in I_3$ \eqref{intervals}:  $u_s(x)=-1+
          o(1)$, 
          $p_s(x)=0+ o(1)$, 
\begin{equation}\label{e:I3}
\begin{array}{rcl rcl}
v_s (x)&=& -1 +\dfrac{2\sinh(L-x^{**})}{\sinh(L)}\,\cosh(x)+
          o(1),
\\[2mm]q_s (x)&=& \dfrac{2\sinh(L-x^{**})}{\sinh(L)}\sinh(x)
          + o(1),
  \\[2mm]
w_s (x)&=& -1 +\dfrac{2\sinh\left((L-x^{**})/D\right)}
{\sinh\left(L/D\right)}\,\cosh\left( x/D\right)+ o(1),
\\[2mm]r_s (x)&=&
\dfrac{2\sinh\left((L-x^{**})/D\right)}{D\sinh\left(L/D\right)}
\sinh\left( x/D\right)+ o(1).
\end{array}
\end{equation}

\item For $x\in I_5$ \eqref{intervals}:  $u_s(x)=+1+ o(1)$, 
  $p_s(x)=0+ o(1)$, 
\begin{equation}\label{e:I2}
\begin{array}{rcl rcl}
v_s (x)&=& 1 -\dfrac{2\sinh(x^{**})}{\sinh(L)}\,\cosh(L-x)+
             o(1),
             \\[2mm]
q_s (x)&=& \dfrac{2\sinh(x^{**})}{\sinh(L)}\,\sinh(L-x)+
              o(1),
  \\[2mm] 
w_s (x)&=& 1 -\dfrac{2\sinh\left(x^{**}/D\right)}{\sinh\left(\
             L/D\right)}\,
             \cosh\left((L-x)/D\right)+ o(1),
\\[2mm]r_s (x)&=& \dfrac{2\sinh\left(x^{**}/D\right)}{D\sinh\left(
             L/D\right)}\,\sinh\left((L-x)/D\right)+ o(1).
\end{array}
\end{equation}
	\end{itemize}

\section{Higher order correction terms of the slow-fast periodic solutions}\label{A:HOT}
We compute the next order correction terms of the slow-fast periodic solutions of Theorem~\ref{th.one}, see also section~\ref{S:SF},  
to see how the period comes into play.
To obtain the next order approximation in the slow dynamics,
we first determine the correction to the slow manifold. We write
$\hu^\pm(v;\varepsilon) = \hu^\pm_0(v)+\varepsilon \hu^\pm_1+\mathcal{O}(\varepsilon^2) $
and $\widehat{p} = 0+\varepsilon
\widehat{p}^\pm_1+\mathcal{O}(\varepsilon^2)$, where $\hu_1$ and
$\widehat{p}_1$ are functions of $(v,q,w,r)$.
Substitution into the slow system~\eqref{e:6slow} and
truncating at second order in $\varepsilon$ gives
\begin{eqnarray*}
  \widehat{p}_1 (v,q,w,r)&=& 
                   \frac{d \hu_0(v)}{d v} q = \frac{Aq}{1-3\hu_0^2(v)},\\   
  \hu_1 (v,q,w,r)&=& \frac{B_1w+C_1}{1-3\hu_0^2}.           
\end{eqnarray*}
These expressions are well-defined on $\mathcal{M}_0^\pm$ away from
the singular points at $\hu_0=\pm 1/\sqrt 3$ (i.e., away from
$A=2/3$). To find the slow dynamics, we write
$v_s=\pm \sqrt{1-A}+\varepsilon v^s_1(x) +\mathcal{O}(\varepsilon^2)$,
$w_s=\pm \sqrt{1-A}+\varepsilon w^s_1(x) +\mathcal{O}(\varepsilon^2)$,
$q_s=\varepsilon q^s_1(x) +\mathcal{O}(\varepsilon^2)$,
$r_s=\varepsilon r^s_1(x) +\mathcal{O}(\varepsilon^2)$. This gives a
linear constant coefficient system of ODEs
\begin{eqnarray*}
(v_1^s)_x &=& q_1^s\,,\\
(q_1^s)_x &=& \frac{2(1-A)}{2-3A}\,v_1^s
               +\frac{C_1\pm B_1\sqrt{1-A^2}}{2-3A}\,,\\ 
(w_1^s)_x &=& \dfrac{r_1^s}{D}\,,\\
(r_1^s)_x &=& \frac1D\,\left(w_1^s+\frac{A}{2-3A}\,v_1^s +
               \frac{C_1\pm B_1\sqrt{1-A^2}}{2-3A}\right),  
\end{eqnarray*}
where we used
$\hu_0(\pm\sqrt{1-A}+\varepsilon v_1^s) = \pm\sqrt{1-A}
+\varepsilon \,Av_1^s/(3A-2)+\mathcal{O}(\varepsilon^2) $
and
$\hu_1(\pm\sqrt{1-A}+\varepsilon v_1^s, \\ \pm\sqrt{1-A}+\varepsilon w_1^s) =
(\pm B_1\sqrt{1-A^2}+C_1)/(3A-2)+\mathcal{O}(\varepsilon) $.
The boundary conditions follow from the behaviour of the slow
variables during the fast phase.
Using that Theorem~\ref{th.one} gives
  $v_f(\xi)=\pm\sqrt{1-A}+\mathcal{O}(\varepsilon)$, the calculation
  in
\eqref{eq:delta_fq} gives
that the change in $q$ over the fast interval is given by
\[
\Delta_q^f(\varepsilon) = \varepsilon \int_{-\frac{1}{\sqrt{\varepsilon}}}^{\frac{1}{\sqrt{\varepsilon}}}
\left(\pm\sqrt{1-A}-u_h(\xi;\pm\sqrt{1-A}\right)\,d\xi +
o(\varepsilon),
\]
where $u_h(\xi;\hu_0^\pm)$ is the orbit in the fast system homoclinic
to $\hu_0^\pm=\pm\sqrt{1-A}$. This expression gives the jump in $q_1^s$ during the
fast phase and hence the boundary condition
\[
q^s_1(0^+)-q^s_1(0^-) = \int_{-\infty}^\infty
\left(\pm\sqrt{1-A}-u_h(\xi;\pm\sqrt{1-A}\right)\,d\xi =: J_1. 
\]
In a similar way, the jump in $r_1^s$ can be determined:
\[
r^s_1(0^+)-r^s_1(0^-) =\frac{1}{D} \int_{-\infty}^\infty
\left(\pm\sqrt{1-A}-u_h(\xi;\pm\sqrt{1-A}\right)\,d\xi =\frac{J_1}{D}. 
\]
As $q^s_0=0$ and $r^s_0=0$ during the fast phase, it follows that $v_1^s$
and $w_1^s$ do not have a jump during the fast phase.
Solving the system of ODEs with those boundary conditions, we get that
for $x\in I_s$:
\[
v^s_1(x) = -\frac{C_1+\ueq^0B_1}{2(1-A)}+
\frac{J_1}{2M}\,\frac{\cosh\left(M  (x\pm L)\right)}{\sinh(ML)}; \quad 
q^s_1(x) = \frac{J_1}{2}\,\frac{\sinh\left(M  (x\pm L)\right)}{\sinh(ML)};
\]
\begin{align*}
w^s_1(x)=&  -\frac{C_1+\ueq^0B_1}{2(1-A)} + 
\frac {J_1N\cosh \left( M(x\pm L) \right) }{2M \left( {D}^{2}{M}^{2}-1
 \right) \sinh \left( ML \right) } \\ &
 \qquad \qquad
+\frac {J_1 \left( \left( {M}^{2}+N \right) {D}^{2} -1\right) }
{2D(1-{D}^{2}{M}^{2}) }
\frac{\cosh \left( (x\pm L)/D \right)}
{  \left( \sinh \left( L/D \right)  \right)}
;
\end{align*}
and 
\[
r^s_1(x) =\frac{J_1}{2D(D^2M^2-1)}\,
\left(\frac {\sinh \left( M(x\pm L) \right) N{D}^{2} }{ \sinh \left( ML \right) }
-\frac{\left(\left( {M}^{2}+N \right) {D}^{2} -1\right)  \sinh \left( (x\pm L)/D \right)}{\sinh \left( L/D \right) }
 \right)
\]
with $M=\sqrt{(2(1-A))/(2-3A)}$, $N=A/(2-3A)$ and
$\ueq^0=\pm\sqrt{1-A^2}$ (the sign in this  expression is not related to
the sign in $I_s$, 
but it is the sign of the base point $\ueq^0$). 
These calculations break down for $A$ near $2/3$ as $M$ and $N$
start diverging. They also break down for $A$ near~0 as the integral $J_1$
will diverge due to the homoclinic undergoing a heteroclinic
bifurcation. These above expressions determine the relation between the profiles and the periodicity of the $2L$-periodic slow-fast solutions of Theorem~\ref{th.one}. 

\section{Proof of Lemma~\ref{lem.melnikov_overview}}\label{S:app_mel}
Before we prove Lemma~\ref{lem.melnikov_overview}, we
analyse the Melnikov function 
$M(z)
$ \eqref{e:jump}.
\begin{lemma}\label{lem.melnikov}
  Define the function $M:[-L,L]\to \mathbb{R}$ as
  $$M(z)  := A_1\,\dfrac{\sinh(z)}{\sinh(L)} +
B_1\,\dfrac{\sinh(z/D)}{\sinh(L/D)}.$$ This function has the following
properties.
\begin{enumerate}
\item The function $M(z)$ is odd and $M(L)=-M(-L)=A_1+B_1$. 
\item Define $\widetilde D = \dfrac{D\sinh(L/D)}{\sinh(L)}$ and
  $\widehat D =\dfrac{D\tanh(L/D)}{\tanh(L)}$, then
  $0<\widetilde D<1$ and $\widehat D>1$. 
 \begin{itemize} \item If
  $ \min(-\widetilde D A_1, -\widehat D
  A_1)\leq B_1\leq\max(-\widetilde D A_1, -\widehat D A_1)$,
  then $M(z)$ is non-monotonic with two turning points in
  $[-L,L]$. If $B_1=-A_1 \widetilde D$, then the turning points coincide
  at $z=0$ and if $B_1=-A_1 \widehat D$, then the turning point
  is at $z= \pm L$. \item Otherwise $M(z)$ is monotonic for $z\in[-L,L]$.
  \end{itemize}
\item 
In the $A_1$-$B_1$ plane, on the curves parametrised
  by~$z\in [-L,0)\cup(0,L] $ as 
\begin{eqnarray*}
A_1^*(z) =C_1\,\frac{\sinh(L)}{\cosh(z)(D\tanh(z/D)-\tanh(z))}, \\
B_1^*(z) =  -C_1\frac{D\sinh(L/D)}{\cosh(z/D)(D\tanh(z/D)-\tanh(z))}, 
\end{eqnarray*}
we have $M(z)|_{(A_1,B_1)=(A_1^*(z),B_1^*(z))}=-C_1$ and
$M'(z)|_{(A_1,B_1)=(A_1^*(z),B_1^*(z))}=0$, hence the Melnikov condition~\eqref{e:jump} has a double root $x^{**}$.
Furthermore, $\dfrac{B_1^*(z)}{A_1^*(z)} = 
-\dfrac{\widetilde D\cosh(z)}{\cosh(z/D)}$, an even function, monotonically decreasing
function for $z\in[0,L]$, and 
$A_1^*(z)=C_1\dfrac{3D^2\sinh(L)}{z^3(D^2-1)}$\\$\left(1+\mathcal{O}(z^2)
\right)$, $B_1^*(z) = -\widetilde{D}A_1^*(z)\left(1+\mathcal{O}(z^2)
\right)$, for $z\to 0$, and
$A_1^*(\pm L)=\pm\dfrac{C_1}{\widehat D -1}$, $B_1^*(\pm L)=
\mp\dfrac{C_1\widehat D}{\widehat D -1} = \mp C_1 -A_1^*(\pm L)$.
\end{enumerate}
\end{lemma}
\begin{proof}[Proof of Lemma~\ref{lem.melnikov}]
  The first observation follows by inspection. To show the second
  observations, we first define the functions $f(L) = D\sinh(L/D)-\sinh(L)$ and
  $g(L) = D\tanh(L/D)-\tanh(L)$. Differentiation shows $f'(L)<0$ and
  $g'(L)>0$ for $L>0$. Since $f(0)=0=g(0)$, this implies that 
  $\widetilde D<1$ and $\widehat D>1$.
  
Next, we differentiate $M$ and find
  \[
M'(z) =  A_1\,\frac{\cosh(z)}{\sinh(L)} +
B_1\,\frac{\cosh(z/D)}{D\sinh(L/D)} =
\frac{\cosh(z)}{D\sinh(L/D)}\,
\left[A_1\widetilde D +B_1 \,\frac{\cosh(z/D)}{\cosh(z)}
\right].
\]
Thus $M'(z)=0$ if and only if $-\dfrac{A_1\widetilde D}{B_1} = \dfrac{\cosh(z/D)}{\cosh(z)}$. 
Since $D>1$, the function $\dfrac{\cosh(z/D)}{\cosh(z)}$ is even and
monotonically decreasing for $z>0$ as

\begin{align*}
  \frac{d}{dz}\left(\frac{\cosh(z/D)}{\cosh(z)}\right) &=
\dfrac{\cosh(z)\sinh(z/D)-D \cosh(z/D)\sinh(z)}{D\cosh^2(z)} 
\\&= \dfrac{\tanh(z/D)-D \tanh(z)}{D\cosh(z)\cosh(z/D)}
<0.
\end{align*}
A quick calculation shows that
$\dfrac{\cosh(L/D)}{\cosh(L)}=\dfrac{\widetilde D}{\widehat D}<1$. Hence,
$\dfrac{\cosh(L/D)}{\cosh(L)} \in \left[\dfrac{\widetilde D}{\widehat D},1\right]$. 
Thus 
$M'(z)$ has exactly one zero in $[0,L]$ if
$-\dfrac{A_1\widetilde D}{B_1} \in \left[\dfrac{\widetilde
    D}{\widehat D},1\right]$ and no zeros in $[0,L]$
otherwise. Rewriting this relation between $A_1$ and $B_1$ gives
the condition in the Lemma.

The statements about the functions $A_1^*(z)$ and $B_1^*(z)$ can
be verified by substitution in the expressions for $M(z)$  and $M'(z)$. 
\end{proof}

Now we are ready to prove Lemma~\ref{lem.melnikov_overview}.
\begin{proof}[Proof of Lemma~\ref{lem.melnikov_overview}]
  Again we distinguish between $C_1=0$ and $C_1\neq 0$.
  \begin{itemize}
  \item Assume $C_1=0$. The Melnikov condition~\eqref{e:jump}
    becomes $M(2x^{**}-L) =0$ for some $x^{**}\in(0,L)$. Point (1) in
    the lemma above shows that this equation is always satisfied at
    $2x^{**}-L=0$, i.e., $x^{**}=L/2$. Combining points (1) and (2)
    gives the remaining statements.
    \item Assume $C_1\neq0$. The saddle-node curve is derived in
      point (3) from the lemma above and the other curves follow from
      points (1) and (2). 
    \end{itemize}\vspace*{-8mm}
  \end{proof}

 The bifurcation diagrams along the dashed curves in the left panel of
 Figure~\ref{f:exist_ab_plane} are shown in
 Figure~\ref{f:bifurcation_a}.
 \begin{figure}[htb]
   \begin{center}
   \includegraphics[width=1\linewidth]{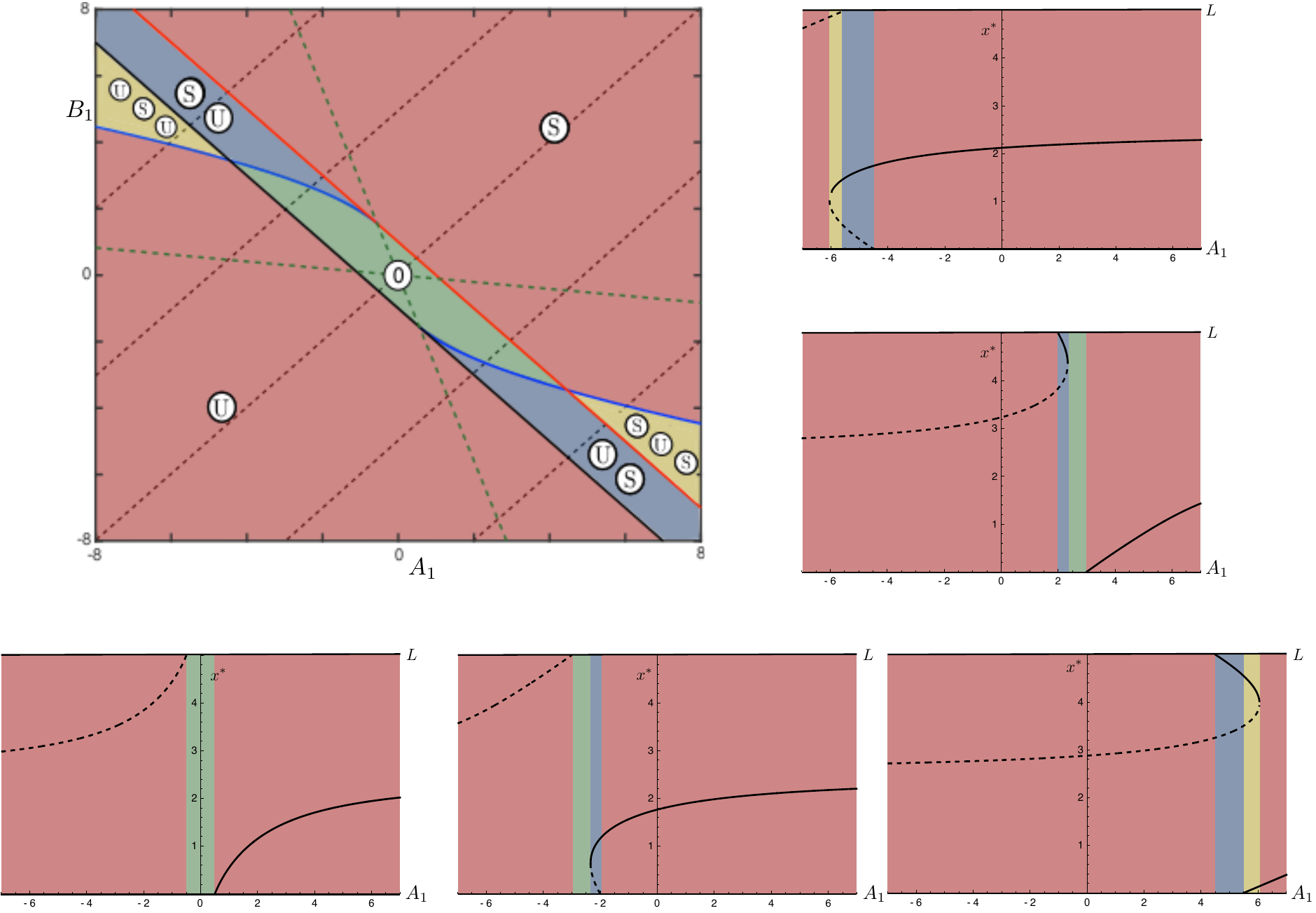}
    \end{center}
 \caption{The bifurcation diagrams for the solutions depicted in the
   right panel of Figure~\ref{f:exist_ab_plane} along the dashed black
   lines. First the right panel of Figure~\ref{f:exist_ab_plane} is
   repeated. Then going from left to right, the relation between
   $A_1$ and $B_1$ in the bifurcation diagrams is:
   $A_1 - B_1 = -10$; $A_1 - B_1 = -5$;
   $A_1 - B_1 = 0$; $A_1 - B_1 = 5$; $A_1 - B_1 =
   10$. In the bifurcation diagrams, the solid black lines correspond
   to stable solutions, the dashed black lines to unstable
   solutions.}\label{f:bifurcation_a}
 \end{figure}

\bibliographystyle{siamplain}
\bibliography{RD3}

\end{document}